\documentclass{amsart}

\usepackage{graphicx}
 \newtheorem{alg}{Algorithm}[section]

 \newtheorem{thm}{Theorem}[section]
 \newtheorem{cor}[thm]{Corollary}
 \newtheorem{lem}[thm]{Lemma}
 \newtheorem{prop}[thm]{Proposition}
 \theoremstyle{definition}
 \newtheorem{defn}[thm]{Definition}
 \theoremstyle{remark}

 \numberwithin{equation}{section}

\begin{document}

%
%
%
%
%
%
%
%
%

\newcommand*\HYPERskip{&}
\catcode`,\active
\newcommand*\pPq{
\begingroup
\catcode`\,\active
\def ,{\HYPERskip}%
\doHyper
}
\catcode`\,12
\def\doHyper#1#2#3#4#5{%
\, _{#1}\Phi_{#2}\left[\begin{matrix}#3 \smallskip \\  #4\end{matrix} \; ; \; #5\right]%
\endgroup
}

\newcommand{\nkand}[2]{{#1}_1 & {#1}_2 & \dots & {#1}_{#2}}
\newcommand{\nklist}[2]{{#1}_1, {#1}_2, \dots, {#1}_{#2}}
\newcommand{\nkplus}[2]{{#1}_1 + {#1}_2 + \cdots + {#1}_{#2}}
\newcommand{\nkprod}[2]{{#1}_1  {#1}_2  \cdots  {#1}_{#2}}
\newcommand{\nkprex}[2]{{#1}_1^{m_1}  {#1}_2^{m_2}  \cdots  {#1}_{#2}^{m_{#2}}}
\newcommand{\nktext}[2]{{${#1}_1$, ${#1}_2$, \dots, ${#1}_{#2}$}}
\newcommand{\bcklist}[3]{{#1}_1, {#2}_1, {#1}_2, {#2}_2, \dots, {#1}_{#3}, {#2}_{#3}}
\newcommand{\bckand}[3]{{#1}_1& {#2}_1& \dots & {#1}_{#3}& {#2}_{#3}}
\newcommand{\abcklist}[4]{{\tfrac{#1}{#2_1}}, {\tfrac{#1}{#3_1}}, \dots, {\tfrac{#1}{#2_{#4}}}, {\tfrac{#1}{#3_{#4}}}}
\newcommand{\abckand}[4]{{\frac{#1}{#2_1}} & {\frac{#1}{#3_1}} & \dots & {\frac{#1}{#2_{#4}}} & {\frac{#1}{#3_{#4}}}}

\newcommand{\ovr}{\frac{(-q;q)_\infty}{(q;q)_\infty}}
\newcommand{\tovr}{\tfrac{(-q;q)_\infty}{(q;q)_\infty}}
\newcommand{\aqprod}[3]{({#1};{#2})_{#3}}
\newcommand{\os}{\overline{\sigma}}
\newcommand{\oo}{\overline}

\newcommand{\abarray}{\begin{pmatrix}
\nkand{\alpha}{k}\\
\nkand{\beta}{k}
\end{pmatrix}}

\newcommand{\Rk}{\overline{R[k]}(z,q)}
\newcommand{\RkII}{\overline{R[2k]}(z,q)}
\newcommand{\Ck}{\overline{R[k]}(z,q)}
\newcommand{\Frobk}{\overline{R}_k(\nklist{x}{k};q)}
\newcommand{\FrobkII}{\overline{R2}_k(\nklist{x}{k};q)}

\newcommand{\frobarray}{\begin{pmatrix}
\lambda_1 &  \lambda_2 &  \cdots  & \lambda_{k}\\
\mu_1 &  \mu_2 &  \cdots  &  \mu_k\\
\end{pmatrix}}

\newcommand{\CC}{\mathbb C}
\newcommand{\RR}{\mathbb R}
\newcommand{\QQ}{\mathbb Q}
\newcommand{\ZZ}{\mathbb Z}
\newcommand{\FF}{\mathbb F}
\newcommand{\NN}{\mathbb{N}_0}

\newcommand{\ninf}{\mathbb{N}_\infty}

\title[Buffered Frobenius Representations]
 {Two Families of Buffered Frobenius Representations of Overpartitions}

\author[Morrill]{Thomas Morrill}

\address{%
Kidder Hall 368\\
Oregon State University\\
Corvallis, OR, 97331\\
United States}

\email{morrillt@math.oregonstate.edu}




\date{\today}

\begin{abstract}
We generalize the generating series of the Dyson ranks and $M_2$-ranks of overpartitions  to obtain $k$-fold variants, and give a combinatorial interpretation of each.
The $k$-fold generating series correspond to the \emph{full ranks} of two families of \emph{buffered Frobenius representations}, which generalize Lovejoy's first and second Frobenius representations of overpartitions, respectively.
\end{abstract}

\maketitle

\section{Introduction and Statement of Results}

A \emph{partition}  of $n$ is a nonincreasing sequence of integers $\lambda =(\nklist{\ell}{k})$ such that the sum of the $\ell_i$ equals $n$.
Each of the $\ell_i$ is called a \emph{part} of $\lambda$.
We use the term \emph{partition statistic} loosely to refer to any integer valued function on the set of partitions.
For example, the \emph{weight} of an arbitrary partition $\lambda$ is the sum of its parts,
\begin{align*}
|\lambda| := \sum_{i=0}^k \ell_i.
\end{align*}
We   use $\ell(\lambda)$ to denote the largest part of $\lambda$, and $\#(\lambda)$ to denote the number of parts of $\lambda$.

Historically, the theory of partition ranks was developed to give combinatorial evidence for the Ramanujan congruences, which state that for all $n\geq 0$,
		\begin{align}
			p(5n+4) &\equiv 0 \pmod{5}, \label{eqn:R5}\\
			p(7n+5) &\equiv 0 \pmod{7}, \label{eqn:R7}\\
			p(11n+6) &\equiv 0 \pmod{11}, \label{eqn:R11}
		\end{align}
where $p(n)$ denotes the number of partitions of $n$.
Given a partition $\lambda$, Dyson \cite{Dyson} defined the \emph{rank} of $\lambda$ to be
\begin{align*}
	r(\lambda) := \ell(\lambda) - \#(\lambda),
\end{align*}
that is,
the largest part of $\lambda$ minus the number of parts of $\lambda$. For example, the partitions of $4$ are given with their ranks in Table \ref{tbl:partns4}.
Note that $p(4)=5$, which agrees with \eqref{eqn:R5}.

		\begin{table}[h]
		\begin{tabular}{|l|c|c|c|c|c|}
			\hline
			$\lambda$  & $(4)$ & $(3,1)$ & $(2,2)$ & $(2,1,1)$ & $(1,1,1,1)$\\
			\hline
			$r(\lambda)$ & $3$ & $1$ & $0$ & $-1$ & $-3$\\
			\hline
		\end{tabular}
		\caption{Ranks of the partitions of $4$. \label{tbl:partns4}}
		\end{table}

 Moreover, each equivalence class of $\ZZ/5\ZZ$ appears exactly once in the second row of Table  \ref{tbl:partns4}.
Atkin and Swinnerton-Dyer \cite{ASDrank} proved that for all $n\geq 0$ and all $i, j \in \ZZ$,
		\begin{align}\label{cong}
			N(i,5n+4,5) = N(j,5n+4,5),
		\end{align}
where  $N(m,n,k)$ denotes the number of partitions of $n$ with rank $m$ modulo $k$.
Consequently, the set of partitions of  $5n+4$ can be separated into five classes of equal size by their ranks, which  proves \eqref{eqn:R5}  via a counting argument.
Atkin and Swinnerton-Dyer also proved that
\begin{align*}
	N(i,7n+5,7) = N(j,7n+5,7),
\end{align*}
 which  treats \eqref{eqn:R7} similarly.
However, it is easy to confirm that
\begin{align*}
	N(i,11n+6,11) = N(j,11n+6,11)
\end{align*}
does not even hold for $n = 0$.
A counting argument for \eqref{eqn:R11} was later found by using the partition \emph{crank} function, which was predicted by Dyson \cite{Dyson} and later defined by Andrews and Garvan \cite{AGCrank}.

We now generalize. An \emph{overpartition} is a nonincreasing sequence of positive integers $\lambda =(\nklist{\ell}{k})$, where the first occurrence of each part may be overlined. 
For example, the fourteen overpartitions of $4$ are given by
		\begin{align*}
		\begin{array}{ccccc}
			(4) &(\overline{4})
			&(3,1) & (\overline{3},\overline{1})
			 &(3, \overline{1}) \\ (\overline{3} , 1)
			&(2,2) & (\overline{2},2)
			&(2,1,1)& (\overline{2},\overline{1},1 )\\
			 (2 , \overline{1},1) &( \overline{2},1,1)
			&(1,1,1,1) & (\overline{1},1,1,1).
		\end{array}
		\end{align*}
Since every partition is an overpartition, we retain the notation $|\lambda|$, $\ell(\lambda)$, and $\#(\lambda)$ for the weight, largest part, and number of parts of an overpartition $\lambda$, respectively.

It is useful to represent partitions or  overpartitions graphically as  arrays of boxes. The \emph{Young tableau} of a partition or overpartition $\lambda = (\nklist{\ell}{k})$ is a left aligned array where the $i$th row of the array consists of $\ell_i$ boxes.
For overpartitions, if the first occurrence of the integer $\ell$ is overlined in $\lambda$, then we mark the last row of $\ell$ boxes with a dot\footnote{This convention ensures that mirroring the diagram across its main diagonal will produce the Young tableau of another overpartition, more commonly known as \emph{conjugating} the overpartition.}.
An example is given in Figure $\ref{Young}$.

\begin{figure}[h]
\includegraphics[scale=.5]{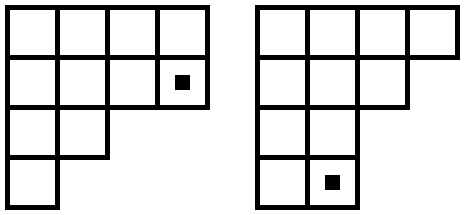}
\label{Young}
\caption{The Young tableau for $(\overline{4},4,2,1)$ and its conjugate, $(4, 3, \overline{2}, 2)$.}
\end{figure}

Because these objects generalize partitions, it is natural to ask if partition statistics can be extended to overpartitions in a meaningful way. We begin by recapping some results for overpartition ranks. The full proofs are given in work of Lovejoy \cite{LovejoyD} \cite{LovejoyM2}.

The \emph{Dyson rank} of an overpartition $\lambda$ is defined to be
\begin{align*}
	\oo{r}_D(\lambda) = \ell(\lambda) - \#(\lambda),
\end{align*}
 an extension of Dyson's rank function for ordinary partitions. For example, if $\lambda = (\overline{4},4,2,1)$, then $\oo{r}_D(\lambda)=0$. We see the generating series for the Dyson ranks of overpartitions in the following theorem.

\begin{thm}[Lovejoy \cite{LovejoyD}] \label{L1thm}
	The coefficient of $z^mq^n$ in the series
	\begin{align} \label{L1}
	\overline{R[1]}(z,q) := \ovr \left( 1 + 2 \sum_{n \geq 1} \frac{(1-z)(1-z^{-1})(-1)^nq^{n^2+n}}{(1-zq^{n})(1-z^{-1}q^{n})} \right)
	\end{align}
	is equal to the number of overpartitions $\lambda$ with $|\lambda| = n$ and $\oo{r}_D(\lambda) = m$.
\end{thm}

Lovejoy  also developed an $M_2$-rank for overpartitions \cite{LovejoyM2}, which expands on Berkovich and Garvan's $M_2$-rank for ordinary partitions whose odd parts cannot repeat \cite{Berk}. Given an overpartition $\lambda = (\nklist{\ell}{k})$, the \emph{$M_2$-rank} of $\lambda$ is defined to be
$$
	\oo{r}_{M_2}(\lambda) :=
	\left\lceil \frac{ \ell(\lambda)}{2} \right\rceil - \#(\lambda) + \# (\lambda_o) -\chi(\lambda),
$$
where $\lambda_o$ is the subpartition of $\lambda$ consisting of all non-overlined odd parts of $\lambda$, and $\chi(\lambda)$ is defined to be
\begin{align*}
\chi(\lambda) := \begin{cases}
1, &\text{if the largest part of } \lambda  \text{ is both odd and non-overlined} \\
0, &\text{otherwise.}
\end{cases}
\end{align*}

For example, let  $\lambda = ( \overline{2},1,1)$. Then $\lambda_o = (1,1)$, and we see that $\oo{r}_{M_2}(\lambda) = 1-3+2-0 = 0$. We see the generating series for the $M_2$-ranks of overpartitions in the following theorem.

\begin{thm}[Lovejoy \cite{LovejoyM2}] \label{L2thm}
	The coefficient of $z^mq^n$ in the series
	\begin{align} \label{L2}
\overline{R[2]}(z;q):=
\ovr \left( 1 + 2 \sum_{n \geq 1} \frac{(1-z)(1-z^{-1})(-1)^nq^{n^2+2n}}{(1-zq^{2n})(1-z^{-1}q^{2n})} \right)
\end{align}
	is equal to the number of overpartitions $\lambda$ with $|\lambda| = n$ and $\oo{r}_{M_2}(\lambda) = m$.
\end{thm}

The proofs of these theorems are based on Lovejoy's first and second Frobenius representations for overpartitions \cite{LovejoyD} \cite{LovejoyM2}, which we summarize in Section \ref{prelim}.
Note the similarity in the summands in \eqref{L1} and \eqref{L2}; they are identical apart from the exponents of $q$ in the summation.

We now continue this pattern. For $k \geq 1$, define the series
\begin{align}\label{Rk}
	\overline{R[k]}(z,q) := \ovr \bigg( 1+ 2 \sum_{n=1}^\infty \frac{(1-z)(1-z^{-1})(-1)^nq^{n^2+kn}}{(1-zq^{kn})(1-z^{-1}q^{kn})} \bigg).
\end{align}
It is natural to ask is if $\Rk$ can be interpreted as the generating series of an overpartition rank.
In this paper we give a partial answer in terms of Frobenius representations.
We may think of a Frobenius representation as an array
\begin{align*}
\begin{pmatrix}
\nkand{a}{k}\\
\nkand{b}{k}
\end{pmatrix},
\end{align*}
where $\lambda = (\nklist{a}{k})$ and $\mu = (\nklist{b}{k})$ are partitions or overpartitions.
As we will see in Section \ref{prelim}, certain Frobenius representations correspond bijectively to overpartitions.

In Section \ref{sec:buff}, we introduce \emph{buffered Frobenius representations},
which are arrays of the form 
\begin{align*}
\begin{pmatrix}
	\nkand{\alpha}{k}\\
	\nkand{\beta}{k}
\end{pmatrix},
\end{align*}
where each of the entries $\alpha_i$ and $\beta_i$ are partitions or overpartitions.
A buffered Frobenius representation can be interpreted as an exploded Young tableau for an ordinary Frobenius representation $(\lambda, \mu)^T$.
Thus, every overpartition admits multiple buffered Frobenius representations.

We now present our first main result, which interprets $\Rk$ in terms of buffered Frobenius representations.

\begin{thm} \label{thm1}
	Let $\zeta_k$ be a primitive $k$th root of unity.
	The coefficient of $z^{\tfrac{m}{k}} q^n$ in $\Rk $ is equal to the weighted count of buffered Frobenius representations of the first kind $\nu$ with at most $k$ columns, $|\nu|=n$, and full rank $m$, where the count is weighted by 
	\begin{align*}
		(-1)^{h(\nu)} \prod_{i=1}^k\zeta_k^{ (i-1)\rho_1^i(\nu)}.
	\end{align*}
	In particular, the count vanishes for buffered Frobenius representations whose full rank is not a multiple of $k$.
\end{thm}


Following Lovejoy's work on the $M_2$-rank and the second Frobenius representation of an overpartition \cite{LovejoyM2},  our second main result interprets $\RkII$ in terms of a second family of buffered Frobenius representations.

\begin{thm} \label{thm2}
	Let $\zeta_k$ be a primitive $k$th root of unity.
	The coefficient of $z^{\tfrac{m}{k}} q^n$ in $\RkII$ is equal to the weighted count of buffered Frobenius representations of the second kind $\nu$ with at most $k$ columns, $|\nu|=n$, and full rank $m$, where the count is weighted by 
	\begin{align*}
		(-1)^{h(\nu)} \prod_{i=1}^k\zeta_k^{(i-1)\rho_2^i(\nu)}.
	\end{align*}
	In particular, the count vanishes for buffered Frobenius representations whose full rank is not a multiple of $k$.
\end{thm}


Each of these families is equipped with $k$ rank functions, $\rho_1^i(\nu)$ and $\rho_2^i(\nu)$, respectively, and $k$ rank-reversing conjugation maps, which are developed in Sections \ref{first} and \ref{second}.
The observant reader will note that $\Rk$ and $\RkII$ are generating series for the ranks of buffered Frobenius representations, rather than for the ranks of overpartitions.
We discuss this gap and the potential for improvement in Section \ref{end}.

The organization of this paper is as follows.
In Section \ref{prelim}, we outline our $q$-series techniques and summarize the motivating results for the Dyson rank and $M_2$-rank.
In Section \ref{sec:buff}, we define a generic buffered Frobenius representation and give a combinatorial map from buffered Frobenius representations to generalized Frobenius representations. 
This allows us to construct our first family of buffered Frobenius representations and prove Theorem \ref{thm1} in Section \ref{first}.
Then, in Section \ref{second}, we construct our second family of buffered Frobenius representations and prove Theorem \ref{thm2}.
Finally, we give our closing remarks in Section \ref{end}.

\section{Preliminaries} \label{prelim}

\subsection{The $q$-Pochhammer Symbol and $q$-Hypergeometric Series.}

We begin with  the definition of the \emph{$q$-Pochhammer symbol} and its conventional shorthand notations. For $a \in \CC$, define
\begin{align}
\label{Pochhammer}
\aqprod{a}{q}{n} &:= \prod_{i=0}^{n-1}(1-aq^i)\\
\aqprod{a}{q}{\infty} &:= \prod_{i=0}^{\infty}(1-aq^i)\\
\aqprod{\nklist{a}{k}}{q}{n} &:= \aqprod{a_1}{q}{n} \aqprod{a_2}{q}{n} \cdots \aqprod{a_k}{q}{n} \\
\aqprod{\nklist{a}{k}}{q}{\infty} &:= \aqprod{a_1}{q}{\infty} \aqprod{a_2}{q}{\infty} \cdots \aqprod{a_k}{q}{\infty}.
\end{align}
Manipulating $q$-Pochhammer symbols typically entails expanding  the product  and canceling individual factors, as seen in the following lemma.

\begin{lem} \label{reduce}
For all nonnegative integers $m$ and $n$,
\begin{align*}
	\frac{\aqprod{a}{q}{m}}{\aqprod{aq}{q}{m+n}} =
	 \frac{(1-a)}{\aqprod{aq^m}{q}{n+1}}
\end{align*}
\end{lem}

\begin{proof}
The case $m = 0$,
\begin{align*}
	\frac{1}{\aqprod{aq}{q}{n}} =
	 \frac{(1-a)}{\aqprod{a}{q}{n+1}},
\end{align*}
 is trivial.

Next, consider $m > 0$.
By expanding the $q$-Pochhammer symbol and canceling like terms, we have
\begin{multline*}
\frac{\aqprod{a}{q}{m}}{\aqprod{aq}{q}{m+n}} = \frac{(1-a)\cdots(1-aq^{m-1})}{(1-aq)\cdots (1-aq^{m-1})(1-aq^m) \cdots (1-aq^{m+n})}\\
 =\frac{(1-a)}{(1-aq^m) \cdots (1-aq^{m+n})}= \frac{(1-a)}{\aqprod{aq^m}{q}{n+1}}.
\end{multline*}
\end{proof}

The $q$-Pochhammer symbol  is necessary for the definition of the $q$-hypergeometric series,
\begin{align}
\label{hype}
\pPq{r}{r-1}{a_1, a_2, a_3, \dots, a_r}{ b_1, b_2, \dots, b_{r-1}}{q; \ z}
:= \sum_{n \geq 0} \frac{\aqprod{\nklist{a}{r}}{q}{n}z^n}{\aqprod{\nklist{b}{r-1},q}{q}{n}}.
\end{align}
These series admit many beautiful transformation formulas; see Gasper and Rahman \cite{GasperRakhman} for  examples. In this paper, we only require Andrews' $k$-fold generalization of the Watson-Whipple transformation.

\begin{thm}[Andrews \cite{AndrewsWW}] \label{Andrews}  Let $a, \bcklist{b}{c}{k}$ be complex numbers, and let $k \geq 1$ and $N\geq 0$.
	Then,
	\begin{multline} \label{eq:Andrews}
		\pPq{2k+4}{2k+3}{a , a^{\tfrac{1}{2}} q  , -a^{\tfrac{1}{2}} q , \bcklist{b}{c}{k}, q^{-N}}{  a^{\tfrac{1}{2}} , -a^{\tfrac{1}{2}} , \abcklist{aq}{b}{c}{k} , aq^{N+1}}{q; \ \frac{a^kq^{k+N}}{b_1c_1\cdots b_k c_k} } \\
		=
		\frac{\aqprod{aq,\tfrac{aq}{b_k c_k}}{q}{N}}{\aqprod{\tfrac{aq}{b_k},\tfrac{aq}{c_k}}{q}{N}}
		\sum_{{
			n_1,
			\dots ,
			n_{k-1} \geq 0
		}}
		\frac{\aqprod{\tfrac{aq}{b_1c_1}}{q}{n_1}}{\aqprod{q}{q}{n_1}}
		\cdots \frac{\aqprod{\tfrac{aq}{b_{k-1}c_{k-1}}}{q}{n_{k-1}}}{\aqprod{q}{q}{n_{k-1}}}\\
		\times
		\frac{\aqprod{b_2,c_2}{q}{N_1}}{\aqprod{\tfrac{aq}{b_1},\tfrac{aq}{c_1}}{q}{N_1}}
		\frac{\aqprod{b_3,c_3}{q}{N_2}}{\aqprod{\tfrac{aq}{b_2},\tfrac{aq}{c_2}}{q}{N_2}}
		\cdots \frac{\aqprod{b_k,c_k}{q}{N_{k-1}}}{\aqprod{\tfrac{aq}{b_{k-1}},\tfrac{aq}{c_{k-1}}}{q}{{N_{k-1}}}}\\
		\times
		\frac{\aqprod{q^{-N}}{q}{N_{k-1}}}{\aqprod{\tfrac{b_k c_k q^{-N}}{a}}{q}{N_{k-1}}}
		\frac{(aq)^{\nkplus{N}{k-2}}q^{N_{k-1}}}{(b_2 c_2)^{N_1}(b_{3}c_{3})^{N_2}\cdots (b_{k-1}c_{k-1})^{N_{k-2}}},
\end{multline}
where we write $N_0 = 0$ and  $N_i = \nkplus{n}{i}$ for all $i \geq 1$.
\end{thm}

Observe that the left hand side of \eqref{eq:Andrews} is a symmetric function in the variables $\bcklist{b}{c}{k}$. Thus, we may permute the indices of $b_i$ and $c_i$ on the right hand side while leaving the corresponding indices fixed on the left hand side. We map
\begin{align*}
	\begin{matrix}
		1 \mapsto (k-1), & 2 \mapsto (k-2), & \dots, & (k-1) \mapsto 1, & k \mapsto k
	\end{matrix},
\end{align*}
which gives the following corollary to Theorem \ref{Andrews}.
\begin{cor} 
 \label{swerdnA}
	Let $a, \bcklist{b}{c}{k}$ be complex numbers, and let $k \geq 1$ and $N\geq 0$.
	Then,
	\begin{multline*}
		\pPq{2k+4}{2k+3}{a, a^{\tfrac{1}{2}} q, -a^{\tfrac{1}{2}} q , \bcklist{b}{c}{k} , q^{-N}}{ , a^{\tfrac{1}{2}} , -a^{\tfrac{1}{2}} , \abcklist{aq}{b}{c}{k} , aq^{N+1}}{q; \ \frac{a^kq^{k+N}}{b_1c_1\cdots b_k c_k} }\\
		=
		\frac{\aqprod{aq,\tfrac{aq}{b_k c_k}}{q}{N}}{\aqprod{\tfrac{aq}{b_k},\tfrac{aq}{c_k}}{q}{N}}
		\sum_{{
			n_1 ,
			\dots ,
			n_{k-1} \geq 0
		}}
		\frac{\aqprod{\tfrac{aq}{b_{k-1}c_{k-1}}}{q}{n_1}}{\aqprod{q}{q}{n_1}}
		\cdots
		\frac{\aqprod{\tfrac{aq}{b_{1}c_{1}}}{q}{n_{k-1}}}{\aqprod{q}{q}{n_{k-1}}}\\
		\times
		\frac{\aqprod{b_{k-2},c_{k-2}}{q}{N_1}}{\aqprod{\tfrac{aq}{b_{k-1}},\tfrac{aq}{c_{k-1}}}{q}{N_1}}
		\frac{\aqprod{b_{k-3},c_{k-3}}{q}{N_{2}}}{\aqprod{\tfrac{aq}{b_{k-2}},\tfrac{aq}{c_{k-2}}}{q}{N_{2}}} 
		\cdots  \frac{\aqprod{b_{1},c_{1}}{q}{N_{k-2}}}{\aqprod{\tfrac{aq}{b_{2}},\tfrac{aq}{c_{2}}}{q}{N_{k-2}}}
		\frac{\aqprod{b_k,c_k}{q}{N_{k-1}}}{\aqprod{\tfrac{aq}{b_{1}},\tfrac{aq}{c_{1}}}{q}{N_{k-1}}}\\
		\times
		 \frac{\aqprod{q^{-N}}{q}{N_{k-1}}}{\aqprod{\tfrac{b_k c_k q^{-N}}{a}}{q}{N_{k-1}}}
		\frac{(aq)^{\nkplus{N}{k-2}}q^{N_{k-1}}}{(b_{k-2}c_{k-2})^{N_1} (b_{k-3}c_{k-3})^{N_2}\cdots (b_{1}c_{1})^{N_{k-2}}},
	\end{multline*}
	where we write $N_0 = 0$ and  $N_i = \nkplus{n}{i}$ for all $i \geq 1$.
\end{cor}

We now summarize Lovejoy's work on the Dyson rank and $M_2$-rank.

\subsection{Summary of Lovejoy's Work}

 In this context, it is convenient to allow partitions and overpartitions to contain $0$ as a part, such as $\lambda = (3,3,0,0,0)$.
We call these \emph{partitions into nonnegative parts} and \emph{overpartitions into nonnegative parts}, respectively\footnote{When unspecified, the terms partition and overpartition should be taken to mean  partitions and overpartitions into positive parts.}.
The reader may consider this approach as a way for shorter partitions and overpartitions to attain a longer length requirement.
For example, we can admit $(3,3)$ in contexts where  a partition with exactly five parts is required.
This is a common technique when working with generalized Frobenius representations, which we now define.

\begin{defn}[Andrews \cite{AndrewsF}]
Let $\mathcal{A}$ and $\mathcal{B}$ be sets of partitions or overpartitions, possibly into nonnegative parts. A generalized Frobenius representation is a two rowed array 
\begin{align*}
\nu = 
\begin{pmatrix}
a_1& a_2& \dots& a_k\\
b_1& b_2& \dots& b_k
\end{pmatrix}
\end{align*}
\\
where $(a_1, a_2, \dots, a_k) \in \mathcal{A}$, and $(b_1, b_2, \dots, b_k) \in \mathcal{B}$.
\end{defn}
We define the \emph{weight} of a generalized Frobenius representation to be the sum of its entries\footnote{
	Note that Lovejoy uses Andrews' convention $|\nu| = k + \sum (a_i + b_i)$ in his earlier work \cite{LovejoyD}.
	Statements of these results have been adjusted for consistency.
},
\begin{align*}
	|\nu| := \sum_{i=1}^k ( a_i + b_i ).
\end{align*}
For example, 
\begin{align*}
\begin{pmatrix}
6&5&5&2\\
6&4&\overline{0}&0
\end{pmatrix}
\end{align*}
is a generalized Frobenius representation with weight 28. The top row is an ordinary partition, and the bottom row is an overpartition into nonnegative parts.
With the correct choice of sets $\mathcal{A}$ and $\mathcal{B}$, the corresponding Frobenius representations are equivalent to overpartitions, as seen in the following theorem. 

\begin{thm}[Corteel, Lovejoy \cite{Opartns}] \label{Dbi}
There is a bijection between overpartitions $\lambda$ and generalized Frobenius representations $\nu = (\alpha, \beta)^T$ where $\alpha$ is a partition into  distinct parts and $\beta$ is an overpartition into nonnegative parts such that $|\lambda| = |\nu|$.
\end{thm}

Using this bijection, we can define the Dyson rank of $\nu$ to be $\oo{r}_D(\lambda)$.
We see a generating series for the Dyson ranks of Frobenius representations in the following lemma.

\begin{lem}[Lovejoy \cite{LovejoyD}]
The coefficient of $z^mq^n$ in the series
\begin{align*}
\sum_{n=0}^\infty \frac{(-1;q)_nq^{\tfrac{n^2+n}{2}}}{(zq, z^{-1}q;q)_n}
\end{align*}
is equal to the number of generalized Frobenius representations $\nu = (\alpha, \beta)^T$ with  $|\nu| = n$, where $\alpha$ is a partition into  distinct parts and $\beta$ is an overpartition into  nonnegative parts, and $\oo{r}_D(\nu) = m$.
\end{lem}

Thus, Theorem \ref{L1thm} reduces to the following $q$-series transformation.

\begin{lem}[Lovejoy \cite{LovejoyD}] For $z \neq 0$,
\begin{multline} \label{LDgen}
\ovr \bigg( 1+ 2 \sum_{n=1}^\infty \frac{(1-z)(1-z^{-1})(-1)^nq^{n^2+n}}{(1-zq^{n})(1-z^{-1}q^{n})} \bigg)\\ =
\sum_{n=0}^\infty \frac{(-1;q)_nq^{\tfrac{n^2+n}{2}}}{(zq,z^{-1}q;q)_n}.
\end{multline}
\end{lem}

The proof of Lemma \ref{LDgen} involves a limiting case of the $q$-Watson-Whipple transformation, or equivalently, the case $k=2$ in Theorem \ref{Andrews}.
Full details of the transformation may be seen as the case $k=1$ in Section \ref{first}.
We now state the algorithm which produces the bijection in Theorem \ref{Dbi}.

\begin{alg}[Corteel, Lovejoy \cite{Opartns}]
	\label{alg:D}
	\hspace{1cm} \newline
	Input: A Frobenius representation 
	\begin{align*}
		\nu =
	\begin{pmatrix}
		a_1 &a_2 &\ldots &a_k\\
		b_1 &b_2 &\ldots &b_k
		\end{pmatrix}
	\end{align*}
as described in Proposition \ref{Dbi}.

\noindent Output: An overpartition $\lambda$ such that $|\lambda| = |\nu|$.	

\begin{enumerate}
\item Initialize $\lambda_1 = \lambda_2 = \emptyset$.

\item We treat $\lambda_1$ as a partition into $b_k$ nonnegative parts. Delete $b_k$ from $\nu$ and add 1 to each part of $\lambda_1$.

\item Delete $a_k$ from $\nu$. If $b_k$ was overlined, append $a_k$ as a part of $\lambda_1$. Otherwise, if $b_k$ was not overlined, append $a_k$ as a part of $\lambda_2$.

\item Repeat Steps (2) and (3) until all parts of $\nu$ are exhausted.

\item Because $(\nklist{a}{k})$ was a partition into distinct parts, $\lambda_2$ is also a partition into distinct parts. We define the output $\lambda$ to be the overpartition with  non-overlined parts given by $\lambda_1$ and overlined parts given by $\lambda_2$.
\end{enumerate}
\end{alg}

An example of Algorithm \ref{alg:D} is shown in Table \ref{alg2}. Further details may be found in work of Lovejoy \cite{LovejoyD}.

\begin{table}[h]
	\centering
	\begin{tabular}{|c|c|c|c|c|}
	\hline
	Iteration & $\alpha$ & $\beta$ & $\lambda_1$ & $\lambda_2$\\
	\hline
	$0$ & $(3,2,1)$ & $(\overline{4}, 4, \overline{3})$ & $\emptyset$ & $\emptyset$\\
	1 &$(3,2)$&$(\overline{4}, 4)$&$(1,1,1,1)$&$\emptyset$\\
	2 &$(3)$&$(\overline{4})$&$(2,2,2,2)$&$(2)$\\
	3 &$\emptyset$&$\emptyset$&$(3,3,3,3,3)$&$(2)$\\
	\hline
	\end{tabular}
	\caption{\label{alg2} A demonstration of Algorithm \ref{alg:D}. This produces the overpartition $\lambda = (3,3,3,3,3,\overline{2})$.}
\end{table}

	The generating series for the $M_2$-rank involves a second family of Frobenius representations, which appear in the following theorem. 

\begin{thm} [Lovejoy\cite{LovejoyM2}] \label{M2bi}
	There is a bijection between overpartitions $\lambda$ and  generalized Frobenius partitions $\nu = (\alpha, \beta)^T$  where $\alpha$ is an overpartition into odd parts and $\beta$ is a partition into nonnegative parts where odd parts may not repeat such that $|\lambda| = |\nu|$.
\end{thm}

As was the case with the Dyson rank, we can define the $M_2$-rank of $\nu$ to be $\oo{r}_{M_2}(\lambda)$.
We see a generating series for the $M_2$-ranks of Frobenius representations in the following lemma.

\begin{lem}[Lovejoy \cite{LovejoyM2}]
	The coefficient of $z^mq^n$ in the series
	\begin{align*}
	\sum_{n \geq 0} \frac{\aqprod{-1}{q}{2n}q^n}{\aqprod{zq^2, z^{-1}q^2}{q^2}{n}}
	\end{align*}
	is equal to the number of Frobenius representations $\nu = (\alpha, \beta)^T$ with $|\nu| = n$, where $\alpha$ is an overpartition into odd parts and $\beta$ is a partition into nonnegative parts,  and $\oo{r}_{M_2}(\nu)=m$.
\end{lem}

Then Theorem \ref{L2thm} reduces to the following $q$-series transformation.

\begin{lem}[Lovejoy \cite{LovejoyM2}]
	 For $z \neq 0$,
	\begin{multline*}
		\ovr \left( 1 + 2 \sum_{n \geq 1} \frac{(1-z)(1-z^{-1})(-1)^nq^{n^2+2n}}{(1-zq^{2n})(1-z^{-1}q^{2n})} \right)\\
		=\sum_{n \geq 0} \frac{\aqprod{-1}{q}{2n}q^n}{\aqprod{zq^2, z^{-1}q^2}{q^2}{n}}.
	\end{multline*}
\end{lem}

As before, the proof utilizes a limiting case of the $q$-Watson-Whipple transformation. Full details may be seen as the case $k=1$ in Section \ref{second}. We now state the algorithm which gives the bijection in Theorem \ref{M2bi}.

\begin{alg}[Lovejoy \cite{LovejoyM2}]
	\label{alg:M2}
	\hspace{1cm} \newline
	Input: A Frobenius representation 
	\begin{align*}
		\nu =
		\begin{pmatrix}
		\alpha\\
		\beta
		\end{pmatrix} =
		\begin{pmatrix}
		a_1 & a_2 &\ldots &a_k\\
		b_1 & b_2 &\ldots &b_k\\
		\end{pmatrix}
	\end{align*} 
as described in Theorem \ref{M2bi}.

\noindent Output: An overpartition $\lambda$ such that $|\lambda| = |\nu|$.

\begin{enumerate}
\item Initialize $\lambda = \emptyset$.

\item For each odd integer $n<a_1$ which does not appear overlined in $\alpha$, we insert $\overline{n}$ in its correct position in $\alpha$. We also append $-n$ as a part of $\beta$.

\item Reindex the parts of $\beta$ so that from left to right, odd integers appear in increasing order, followed by even integers in decreasing order.

\item For each pair $(a_i, b_i)$, let $\ell_i = a_i+b_i$. If $b_i$ is even, append $\ell_i$ as a part of $\lambda$ with the same overline marking as $a_i$. If $b_i$ is odd, append $\ell_i$ as a part of $\lambda$ with the opposite overline marking as $a_i$. Reindex the $\ell_i$ in non-increasing order, with the convention that $\overline{n} > n$.
\end{enumerate}
\end{alg}

\begin{table}[h]
	\begin{tabular}{|c|l|l|l|}
		\hline
		Step & $\alpha$ & $\beta$ & $\lambda$ \\
		\hline 1&$(5, \overline{1})$&$(6,5)$& $\emptyset$ \\
		2&$(5, \overline{3}, \overline 1)$&$(6,5,-3)$&$\emptyset$ \\
		3&$(5, \overline{3}, \overline 1)$&$(-3, 5, 6)$&$\emptyset$\\
		4&$\emptyset$&$\emptyset$& $(8,\overline{7}, \overline{2})$\\
		\hline
	\end{tabular}
\caption{\label{alg3} Demonstration of Algorithm \ref{alg:M2}.}
\end{table}

\noindent An example of Algorithm 2 is demonstrated in Table \ref{alg3}. The reverse algorithm is a modification of Corteel and Lovejoy's work on vector partitions \cite{CLFrob}. We present it below for completeness. For this algorithm, we let $s(\lambda)$ denote the smallest part of the overpartition $\lambda$.

\begin{alg}[Corteel, Lovejoy
\cite{CLFrob} \cite{LovejoyM2}] \label{alg:2M}
Input: An overpartition $\lambda$.

\noindent Output: A second Frobenius representation $\nu = (\alpha, \beta)^T$ such that $|\nu| = |\lambda|$.

\begin{enumerate}
\item Initialize $\alpha = \beta := \emptyset$ and $a :=1$. Dissect $\lambda$ into four partitions $\overline{\pi}_e$, $ \pi_e$, $\overline{\pi}_o$, and  $\pi_o$ as follows. Let $\overline{\pi}_e$ be the subpartition consisting of all even overlined parts of $\lambda$. Let $\pi_e$ be the subpartition consisting of all even non-overlined parts of $\lambda$. We define $\overline{\pi}_o$ and $\pi_o$ analogously for the odd parts of $\lambda$.

\item If $\overline{\pi}_o = \emptyset$, or if $s(\overline{\pi}_o) \leq  s(\pi_o)$, then append $\overline{a}$ as a part of $\alpha$, append $s(\overline{\pi}_o)-a$ as a part of $\beta$, and delete the smallest part of $\overline{\pi}_o$.

\item Otherwise, append $a$ as a part of $\alpha$, append $s(\pi_o)-a$ as a part of $\beta$, delete the smallest part of $\pi_o$, and set $a := a+2$.

\item Repeat Steps (2) and (3) until both $\overline{\pi}_o$ and $\pi_o$ are exhausted.

\item If $\pi_e = \emptyset$, or if $s(\pi_e) <  s(\overline{\pi}_o)$, then append $a$ as a part of $\alpha$, append $s(\overline{\pi}_e)-a$ as a part of $\beta$, and delete the smallest part of $\overline{\pi}_e$.

\item Otherwise, append $\overline{a}$ as a part of $\alpha$, append $s(\pi_e)-a$ as a part of $\beta$, delete the smallest part of $\pi_e$, and set $a := a+2$.

\item Repeat Steps (5) and (6) until both $\overline{\pi}_o$ and $\pi_o$ are exhausted.

\item If a part $-n$ occurs in $\beta$, delete both $-n$ from $\beta$ and $\overline{n}$ from $\alpha$.

\end{enumerate}
\end{alg}

\noindent An example of Algorithm \ref{alg:2M} is given in Table \ref{2gla}.

This ends our presentation of previous results. We now introduce the notion of buffered Frobenius representations.

\begin{table}[h]
\centering
	\begin{tabular}{|c|l|l|l|l|l|l|l|}
	\hline
	Iteration & $\overline{\pi}_e$ &$\pi_e$ &$\overline{\pi}_o$ &$\pi_o$ & $a$ & $\alpha$ & $\beta$ \\
	\hline
	0 & $(\overline{2})$ & $(8)$ & $(\overline{7})$ & $\emptyset$ & $1$ & $\emptyset$ & $\emptyset$\\
	1 & $(\overline{2})$ & $(8)$ & $\emptyset$ & $\emptyset$ & $3$ & $(\overline{1})$ & $(6)$\\
	2 & $(\overline{2})$ & $\emptyset$ & $\emptyset$ & $\emptyset$ & $5$ & $(\overline{3},\overline{1})$ & $(6,5)$\\
	3 & $\emptyset$ & $\emptyset$ & $\emptyset$ & $\emptyset$ & $5$ & $(5, \overline{3},1)$ & $(6,5,-3)$\\
	4 & $\emptyset$ & $\emptyset$ & $\emptyset$ & $\emptyset$ & $5$ & $(5, \oo{1})$ & $(6,5)$\\
	\hline
	\end{tabular}
\caption{\label{2gla} Demonstration of Algorithm \ref{alg:2M}.}
\end{table}

\section{Buffered Frobenius Representations} \label{sec:buff}

We use the following abbreviated notation for the rest of the paper.
If \nktext{A}{k} and \nktext{B}{k} are sets,
we write
\begin{align*}
	\begin{pmatrix}
	\nkand{\alpha}{k}\\
	\nkand{\beta}{k}
	\end{pmatrix}
	 \in
	\begin{pmatrix}
	\nkand{A}{k}\\
	\nkand{B}{k}
	\end{pmatrix}
\end{align*}
to mean that $\alpha_i \in A_i$ and $\beta_i \in B_i$ for all $1 \leq i \leq k$.

\begin{defn}
Let $\oo{P_0}$ denote the set of overpartitions into nonnegative parts, and let $P_0$ denote the set of partitions into nonnegative parts.
A buffered Frobenius representation is a two rowed array
\begin{align*}
	\nu
	=
	\begin{pmatrix}
	\nkand{\alpha}{k}\\
	\nkand{\beta}{k}
	\end{pmatrix}
	\in
	\begin{pmatrix}
		\oo{P_0} & P_0 & \dots & P_0 \\
		\oo{P_0} & P_0 & \dots & P_0
	\end{pmatrix},
\end{align*}
where for all $i$, we have  $\#(\alpha_i) \geq \#(\alpha_{i+1})$ and $\#(\beta_{i}) = \#(\alpha_i) $. 
Additionally, we may mark either of $\alpha_i$ or $\beta_i$ with a hat if $i<k$.
\end{defn}

The \emph{weight} of a buffered Frobenius representation is defined to be
\begin{align*}
	|\nu| :=
 \sum_{1\leq i \leq k} |\alpha_{i}|+|\beta_{i}|.
\end{align*}
We see that every generalized Frobenius representation as in Section \ref{prelim}
\begin{align*}
	\begin{pmatrix}
		\nkand{a}{k}\\
		\nkand{b}{k}
	\end{pmatrix}
\end{align*}
can be interpreted as a buffered Frobenius representation
\begin{align*}
	\begin{pmatrix}
		\alpha_1\\
		\beta_1
	\end{pmatrix}
	=
	\begin{pmatrix}
		(\nklist{a}{k})\\
		(\nklist{b}{k})
	\end{pmatrix},
\end{align*}
although this only produces simple examples.
The hat notation serves to enrich the combinatorics of buffered Frobenius representations, similar to the purpose of overlining the parts of an overpartition.
For example,
\begin{align}  \label{frobex}
	\begin{pmatrix}
		\alpha_1 & \alpha_2\\
		\beta_1 & \beta_2
	\end{pmatrix}
	=
	\begin{pmatrix}
		\widehat{(3,3,2,1)} & (1, 0, 0 )\\
		(\oo{3},\oo{2},2,2) & (4, 1, 1)
	\end{pmatrix}
\end{align}
is a buffered Frobenius representation.
Note that $\ell(\beta_2) > \ell (\beta_1)$; only the sequences $\{ \#(\alpha_i) \}$ and $\{ \#(\beta_i) \}$ must be nonincreasing.

\subsection{Buffered Young Tableaux}
Given a buffered Frobenius representation
\begin{align*}
	\nu =
	\begin{pmatrix}
		\nkand{\alpha}{k}\\
		\nkand{\beta}{k}
	\end{pmatrix}
\end{align*}
 we construct \emph{buffered Young tableaux} to represent the entries of $\nu$ by using $k$ colors as follows.

First, we  draw the Young tableau for $\alpha_1$ in the first color.
Next, we draw the Young tableau for $\alpha_2$ in the second color.
However, we align the boxes for $\alpha_2$ to the right edge of the tableau for $\alpha_1$.
If $\alpha_1$ is marked with a hat, we shift the tableau for $\alpha_{2}$ to the right by one unit and leave a buffer between $\alpha_1$ and $\alpha_{2}$.
For example, if $\alpha_1 = \widehat{(3, 2, 1)}$ and $\alpha_2 = (2, 2, 1)$, then we produce the tableaux in Figure \ref{fig:B-tab}.

\begin{figure}[h]
	\centering
	\includegraphics[scale=0.4]{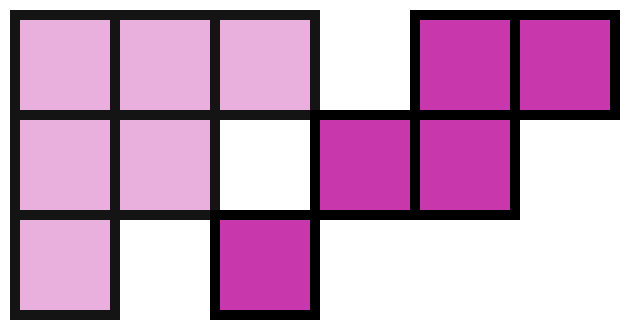}
	\caption{
		\label{fig:B-tab}
		The buffered Young tableaux for $\alpha_1 = \widehat{(3, 2, 1)}$ and $\alpha_2 = (2, 2, 1)$.
	}
\end{figure}

We then continue by drawing the tableau for each $\alpha_i$ in the $i$th color, aligned to the right edge of the preceding tableau, and shifted to the right by one unit if $\alpha_i$ is marked with a hat.
We draw the tableaux for the $\beta_i$ in the same manner.
For example, Figure \ref{frobex-young} shows the  buffered Young tableaux for the buffered Frobenius representation in \eqref{frobex}.

\begin{figure}[h]
	\centering
	\includegraphics[scale=0.4]{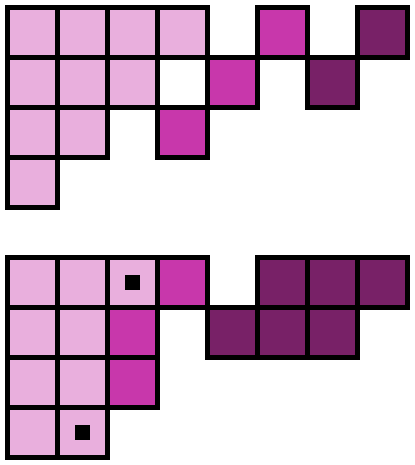}
	\caption{The buffered Young Tableaux for the buffered Frobenius representations in \eqref{frobex}.}
	\label{frobex-young}
\end{figure}

Note that entries marked with a hat increase the width of the tableaux without increasing the number of boxes.
  There are no tableaux which  could indicate a buffer to the right of $\alpha_k$ or $\beta_k$, which corresponds to the restriction that neither $\alpha_k$  or $\beta_k $ can be marked with a hat.

\subsection{The Jigsaw Map}
Visualizing buffered Frobenius representations by their  tableaux suggests that we should  interpret buffered Frobenius representations as the exploded Young tableaux of generalized Frobenius representations.
To reassemble the generalized Frobenius representation, we use the \emph{jigsaw map}.

Let $\nu$ be a buffered Frobenius representation
\begin{align*}
	\nu = \abarray,
\end{align*}
where for all $i$,
\begin{align*}
	\alpha_i &= (a_{(i,1)}, a_{(i,2)}, \dots, a_{(i,k_i)} )\\
	\beta_i &= (b_{(i,1)}, b_{(i,2)}, \dots, b_{(i,k_i)} ).
\end{align*}
We seek to construct a generalized Frobenius representation
\begin{align*}
	j(\nu)
	=
	\begin{pmatrix}
		\nkand{a}{k_1}\\
		\nkand{b}{k_1}
	\end{pmatrix}
	,
\end{align*}
where $(\nklist{a}{k_1})$ and $(\nklist{b}{k_1})$ are partitions or overpartitions into nonnegative parts.

First, discard any hats from the entries of $\nu$.
We then rewrite each $\alpha_i$ and $\beta_i$ as a partition into $k_1$ nonnegative parts,
\begin{align*}
	\alpha_i 
	&= (
		\overbrace{
			a_{(i,1)}, a_{(i,2)}, \dots, a_{(i,k_i)}, 0, \dots, 0
		}^{k_1}
	),\\
	\beta_i 
	&= (
		\overbrace{
			b_{(i,1)}, b_{(i,2)}, \dots, b_{(i,k_i)}, 0, \dots, 0
		}^{k_1}
	)
	.
\end{align*}
For all $1\leq j \leq k_1$, we define the integers $a_j$ to be 
\begin{align*}
	a_j &= \sum_{i=1}^{k_1} a_{(i,j)},
	\\
	b_j &= \sum_{i=1}^{k_1} b_{(i,j)}
	.
\end{align*}
Finally, we overline $a_j$ or $b_j$ if and only if the $j$th part of $\alpha_1$ or $\beta_1$ is overlined, respectively\footnote{
This is why only $\alpha_1$ and $\beta_1$ may be overpartitions.
}.
Graphically, this is equivalent to removing the colors from the buffered Young tableaux and aligning the boxes to the left, with careful attention paid to the convention for overlined parts.

We now move away from the generic treatment in order to present Theorem \ref{thm1}.

\section{Buffered Frobenius Representations of the First Kind} \label{first}

In order to apply Corollary \ref{swerdnA} to $\Rk$, we consider the series
\begin{multline} \label{Frobk}
	\Frobk\\ := \ovr \bigg( 1+ 2 \sum_{n=1}^\infty (-1)^nq^{n^2+kn}\prod_{i=1}^k\frac{(1-x_i)(1-x^{-1}_i)}{(1-x_iq^{n})(1-x^{-1}_iq^{n})} \bigg),
\end{multline}
bearing in mind that
$$
	\overline{R}_k(\sqrt[k]{z},\zeta_k\sqrt[k]{z},\dots, \zeta_k^{k-1}\sqrt[k]{z};q) = \overline{R[k]}(z,q).
$$
We see a transformation of $\Frobk$ in the theorem below.

\begin{thm} \label{firsthype}
	Let $k \geq 1$ be a positive integer. Then we have
	\begin{multline*}
		\ovr \bigg( 1+ 2 \sum_{n=1}^\infty (-1)^nq^{n^2+kn}\prod_{i=1}^k\frac{(1-x_i)(1-x^{-1}_i)}{(1-x_iq^{n})(1-x^{-1}_iq^{n})} \bigg)\\
		= \sum_{{
			n_1 ,
			\dots ,
			n_k \geq 0
		}}
		 \aqprod{-1}{q}{N_k}q^{\tfrac{N_k^2-N_k}{2}}
		\prod_{i=1}^{k} \frac{(1-x_{k-i+1})(1-{x^{-1}_{k-i+1}})q^{N_i}}{(x_{k-i+1}q^{N_{i-1}},x_{k-i+1}^{-1}q^{N_{i-1}})_{n_i+1}},
	\end{multline*}
	where we write $N_0 = 0$ and  $N_i = \nkplus{n}{i}$ for all $i \geq 1$.
\end{thm}




\begin{proof}
We begin by substituting $k \mapsto k+1$ into Corollary \ref{swerdnA}. Letting $N \to \infty$  turns the transformation of terminating series into a transformation of infinite series. The left side becomes
\begin{multline*}
	\sum_{n=0}^\infty
	\frac{\aqprod{a,qa^{\tfrac{1}{2}},-qa^{\tfrac{1}{2}},b_1, c_1, \dots, b_{k+1},c_{k+1}}{q}{n}(-1)^n q^{\tfrac{n^2-n}{2}}}{\aqprod{q,a^{\tfrac{1}{2}},-a^{\tfrac{1}{2}},\tfrac{aq}{b_1},\tfrac{aq}{c_1},\dots, \tfrac{aq}{b_{k+1}},\tfrac{aq}{c_{k+1}}}{q}{n}}\\
	\times \left( \frac{a^{k+1}q^{k+1}}{b_1c_1\cdots b_{k+1} c_{k+1}} \right)^n.
\end{multline*}
When $n=0$, the $q$-Pochhammer symbols take their trivial value, and the summand is equal to $1$. For $n >0$, we may simplify the summand using the relation
\begin{align} \label{pochcancel}
\frac{\aqprod{a,qa^{\tfrac{1}{2}},-qa^{\tfrac{1}{2}}}{q}{n}}{\aqprod{a^{\tfrac{1}{2}},-a^{\tfrac{1}{2}}}{q}{n}} = (1-aq^{2n})\aqprod{aq}{q}{n-1}.
\end{align}
Thus the left hand side is equal to
\begin{multline*}
	1+\sum_{n=1}^\infty
	(1-aq^{2n})\aqprod{aq}{q}{n-1}
	\frac{\aqprod{b_1,c_1, \dots, b_{k+1}, c_{k+1}}{q}{n}(-1)^n q^{\tfrac{n^2-n}{2}}}{\aqprod{q,\abcklist{aq}{b}{c}{k+1}}{q}{n}}\\
	\times \left( \frac{a^{k+1}q^{k+1}}{b_1c_1\cdots b_{k+1} c_{k+1}} \right)^n.
\end{multline*}
On the right hand side, we use the relation
\begin{align} \label{Nlim}
	\lim_{N \to \infty} \frac{\aqprod{q^{-N}}{q}{N_k}}{\aqprod{a^{-1}b_{k+1}c_{k+1}q^{-N}}{q}{N_k}} &= \lim_{N \to \infty} \prod_{i=0}^{N_k-1} \frac{(q^N-q^i)}{(q^N-a^{-1}b_{k+1}c_{k+1}q^i)}\\
	&= \prod_{i=0}^{N_k-1} \frac{-q^i}{-a^{-1}b_{k+1}c_{k+1}q^i} = \left(\frac{a}{b_{k+1}c_{k+1}}\right)^{N_k}
\end{align}
to obtain
\begin{multline*}
	\frac{\aqprod{aq,\tfrac{aq}{b_{k+1} c_{k+1}}}{q}{\infty}}{\aqprod{\tfrac{aq}{b_{k+1}},\tfrac{aq}{c_{k+1}}}{q}{\infty}}
		\sum_{{
			n_1 ,
			\dots ,
			n_{k} \geq 0
		}}
	\frac{\aqprod{\tfrac{aq}{b_kc_k}}{q}{n_1}}{\aqprod{q}{q}{{n_1}}} \cdots\frac{\aqprod{\tfrac{aq}{b_1c_1}}{q}{n_k}}{\aqprod{q}{q}{{n_k}}} \\
	\times \frac{\aqprod{b_{k-1},c_{k-1}}{q}{N_1}}{\aqprod{\tfrac{aq}{b_{k}},\tfrac{aq}{c_{k}}}{q}{N_1}} \frac{\aqprod{b_{k-2},c_{k-2}}{q}{N_2}}{\aqprod{\tfrac{aq}{b_{k-1}},\tfrac{aq}{c_{k-1}}}{q}{N_2}} \cdots \frac{\aqprod{b_{1},c_{1}}{q}{N_{k-1}}}{\aqprod{\tfrac{aq}{b_{2}},\tfrac{aq}{c_{2}}}{q}{N_{k-1}}}\\
	\times  \frac{\aqprod{b_{k+1},c_{k+1}}{q}{N_{k}}}{\aqprod{\tfrac{aq}{b_{1}},\tfrac{aq}{c_{1}}}{q}{N_{k}}}	\frac{(aq)^{\nkplus{N}{k}} }{(b_{k}c_{k})^{N_1}  \cdots (b_{1}c_{1})^{N_{k-1}}(b_{k+1} c_{k+1})^{N_k}}.
\end{multline*}
Setting $a=1$, the equation becomes
\begin{multline*}
	1+\sum_{n=1}^\infty
	(1+q^{n})
	\frac{\aqprod{b_1,c_1, \dots, b_{k+1}, c_{k+1}}{q}{n}(-1)^n q^{\tfrac{n^2-n}{2}}}{\aqprod{\abcklist{q}{b}{c}{k+1}}{q}{n}}
	\left( \frac{q^{k+1}}{b_1c_1\cdots b_{k+1} c_{k+1}} \right)^n\\
	=\frac{\aqprod{q,\tfrac{q}{b_{k+1} c_{k+1}}}{q}{\infty}}{\aqprod{\tfrac{q}{b_{k+1}},\tfrac{q}{c_{k+1}}}{q}{\infty}}
			\sum_{{
			n_1, 
			\dots ,
			n_{k-1} \geq 0
		}}
	\frac{\aqprod{\tfrac{q}{b_kc_k}}{q}{n_1}}{\aqprod{q}{q}{{n_1}}} \cdots\frac{\aqprod{\tfrac{q}{b_1c_1}}{q}{n_k}}{\aqprod{q}{q}{{n_k}}} \\
	\times \frac{\aqprod{b_{k-1},c_{k-1}}{q}{N_1}}{\aqprod{\tfrac{q}{b_{k}},\tfrac{q}{c_{k}}}{q}{N_1}} \frac{\aqprod{b_{k-2},c_{k-2}}{q}{N_2}}{\aqprod{\tfrac{q}{b_{k-1}},\tfrac{q}{c_{k-1}}}{q}{N_2}} \cdots \frac{\aqprod{b_{1},c_{1}}{q}{N_{k-1}}}{\aqprod{\tfrac{q}{b_{2}},\tfrac{q}{c_{2}}}{q}{N_{k-1}}} \\
	\times \frac{\aqprod{b_{k+1},c_{k+1}}{q}{N_{k}}}{\aqprod{\tfrac{aq}{b_{1}},\tfrac{q}{c_{1}}}{q}{N_{k}}}	\frac{q^{\nkplus{N}{k}}}{(b_{k}c_{k})^{N_1}  \cdots (b_{1}c_{1})^{N_{k-1}}(b_{k+1} c_{k+1})^{N_k}}.
\end{multline*}
We set $b_i = x_i$, $c_i = x_i^{-1}$ for $1\leq i\leq k$, and $b_{k+1}=-1$.
This cancels the term
\begin{align*}
	\frac{(-1)^n}{b_{k+1}^n}.
\end{align*}
On the left hand side, we use the identity
\begin{align*}
(1+q^n)\frac{\aqprod{-1}{q}{n}}{\aqprod{-q}{q}{n}} = 2,
\end{align*}
and obtain
\begin{align*}
	1+2\sum_{n=1}^\infty
	\frac{\aqprod{x_1,x^{-1}_1,  \dots,  x_k, x_k^{-1}, c_{k+1}}{q}{n} }{\aqprod{x_1q,x^{-1}_1q,  \dots,  x_kq, x_k^{-1}q, c^{-1}_{k+1}q}{q}{n}} \frac{q^{\tfrac{n^2-n}{2}+(k+1)n}}{ c_{k+1}^n}.
\end{align*}
The right hand side becomes
\begin{multline*}
	\frac{\aqprod{q,\tfrac{-q}{ c_{k+1}}}{q}{\infty}}{\aqprod{-q,\tfrac{q}{c_{k+1}}}{q}{\infty}}
		\sum_{{
			n_1,
			\dots ,
			n_{k} \geq 0
		}}
	\frac{\aqprod{x_{k-1},x^{-1}_{k-1}}{q}{N_1}}{\aqprod{x_kq,x_k^{-1}q}{q}{N_1}}\\
	\times \frac{\aqprod{x_{k-2},x^{-1}_{k-2}}{q}{N_2}}{\aqprod{x_{k-1}q,x_{k-1}^{-1}q}{q}{N_2}} \cdots \frac{\aqprod{x_{1},x^{-1}_{1}}{q}{N_{k-1}}}{\aqprod{x_2q,x_2^{-1}q}{q}{N_{k-1}}}\\
	\times  \frac{\aqprod{-1, c_{k+1}}{q}{N_k}}{\aqprod{x_1q,x_1^{-1}q}{q}{N_k}}  \frac{q^{\nkplus{N}{k}}}{ (-c_{k+1})^{N_k}}  .
\end{multline*}

\noindent We now let $c_{k+1} \to \infty$. On the left hand side, we use the simple identities
\begin{align} \label{cklimit1}
\lim_{c_{k+1} \to \infty} \frac{\aqprod{c_{k+1}}{q}{n}}{c_{k+1}^n} &= (-1)^n q^{\tfrac{n^2-n}{2}}\\ \label{cklimit2}
\lim_{c_{k+1} \to \infty} \aqprod{ c^{-1}_{k+1}q}{q}{n} &= 1
\end{align}
to obtain
\begin{align*}
	1+2\sum_{n=1}^\infty
	\frac{\aqprod{x_1, x_1^{-1}, \dots, x_k, x_k^{-1}}{q}{n} (-1)^n q^{n^2+kn}}{\aqprod{x_1q, x_1^{-1}q, \dots, x_kq, x_k^{-1}q}{q}{n}}.
\end{align*}

\noindent On the right hand side,  applying \eqref{cklimit1} and \eqref{cklimit2} produces
\begin{multline*}
	\frac{\aqprod{q}{q}{\infty}}{\aqprod{-q}{q}{\infty}}
		\sum_{{
			n_1,
			\dots ,
			n_{k} \geq 0
		}}
	 \frac{\aqprod{x_{k-1},x^{-1}_{k-1}}{q}{N_1}}{\aqprod{x_kq,x_{k}^{-1}q}{q}{N_1}}
	 \cdots \frac{\aqprod{x_1,x_1^{-1}}{q}{N_{k-1}}}{\aqprod{x_2q,x_{2}^{-1}q}{q}{N_{k-1}}} \\
	\times \frac{ \aqprod{-1}{q}{N_k}}{\aqprod{x_1q,x_1^{-1}q}{q}{N_k}} q^{\nkplus{N}{k-1}+\tfrac{N_k^2+N_k}{2} }.
\end{multline*}
Applying Lemma \ref{reduce} to the left hand side of the equation  and multiplying by $\tovr$ gives us
\begin{align*}
\ovr \bigg( 1+ 2 \sum_{n=1}^\infty (-1)^nq^{n^2+kn}\prod_{i=1}^k\frac{(1-x_i)(1-x_i^{-1})}{(1-x_iq^{n})(1-x_i^{-1}q^{n})} \bigg).
\end{align*}

\noindent On the right hand side of the equation, we use the fact that $N_i = N_{i-1}+n_i$ for all $1 \leq i \leq k$ with  Lemma \ref{reduce} to write
\begin{align*}
 \frac{\aqprod{x}{q}{N_{i-1}}}{\aqprod{xq}{q}{N_i}} = \frac{(1-x)}{\aqprod{xq^{N_{i-1}}}{q}{n_i +1}}.
\end{align*}
Multiplying the right hand side of the equation by $\tovr$ gives
\begin{align*}
		\sum_{{
			n_1, 
			\dots ,
			n_{k} \geq 0
		}}
	\frac{(-1;q)_{N_k}q^{\tfrac{N_k^2-N_k}{2} + {N}_{1}}}{(x_kq,x^{-1}_kq)_{n_1}}
	\left(\prod_{i=2}^{k} \frac{(1-x_{k-i+1})(1-{x^{-1}_{k-i+1}})q^{N_i}}{(x_{k-i+1}q^{N_{i-1}},x^{-1}_{k-i+1}q^{N_{i-1}})_{n_i+1}}\right).
\end{align*}

\noindent Finally, as $N_0 := 0$, we may rewrite the right hand side using
\begin{align*}
\frac{1}{\aqprod{x_kq, {x_k}^{-1}q}{q}{n_1}} = \frac{(1-x_k)(1-x^{-1}_k)}{\aqprod{x_kq^{N_0}, {x^{-1}_kq^{N_0}}}{q}{n_1+1}},
\end{align*}
which gives us the desired equation,
\begin{multline} \label{gen1}
\ovr \bigg( 1+ 2 \sum_{n=1}^\infty (-1)^nq^{n^2+kn}\prod_{i=1}^k\frac{(1-x_i)(1-x^{-1}_i)}{(1-x_iq^{n})(1-q^{n}x^{-1}_i)} \bigg)\\
	=
		\sum_{{
			n_1,
			\dots ,
			n_{k} \geq 0
		}}
	(-1;q)_{N_k}q^{\tfrac{N_k^2-N_k}{2}}
	\prod_{i=1}^{k}
	\frac{(1-x_{k-i+1})(1-{x^{-1}_{k-i+1}})q^{N_{i}}}{(x_{k-i+1}q^{N_{i-1}},x^{-1}_{k-i+1}q^{N_{i-1}})_{n_i+1}}.
\end{multline}

\end{proof}

\subsection{Overpartition Statistics}

In order to interpret \eqref{gen1} as a generating series, we must introduce some partition and overpartition statistics.
The first statistic we consider appears in Franklin's proof of Euler's pentagonal number theorem \cite{Abook}. We will use several variations of this statistic, so we take the opportunity to name it the \emph{bracket} of a partition.

Given a partition $\lambda = (\nklist{\ell}{n})$, the bracket of $\lambda$ is defined to be the length of the longest sequence of the form $(\nklist{\ell}{k})$, where for all $1 \leq i <k$, we have $\ell_i  = \ell_{i+1} +1$. We retain Andrews' notation of $\sigma(\lambda)$ to denote the bracket of $\lambda$.

For example, if $\lambda = (7,6,5,3,2)$, then we consider the sequences
\begin{align*}
\begin{matrix}
(7), &
(7,6), &
(7,6,5),
\end{matrix}
\end{align*}
the longest of which has length three. Therefore, $\sigma(\lambda)=3$.

We see how the partition rank and the partition bracket relate to \eqref{gen1} in the following lemma.

\begin{lem} \label{initrun}
Fix nonnegative integers $1 \leq s\leq t$. The coefficient of $z^mq^n$ in
\begin{align*}
\frac{q^{\tfrac{t^2+t}{2}}}{\aqprod{zq^s}{q}{t-s+1}}
\end{align*}
is equal to the number of partitions $\lambda$ of $n$ into $t$ distinct parts with $\sigma(\lambda)\geq s$ and   $r(\lambda) = m$.
\end{lem}

\begin{proof}
The term 
\begin{align*}
\frac{1}{\aqprod{zq^s}{q}{t-s+1}} = \frac{1}{(1-zq^s)} \frac{1}{(1-zq^{s+1})} \cdots \frac{1}{(1-zq^t)}
\end{align*}
generates the columns of a Young tableau, where $m$ tracks the number of columns generated. The length of these columns is bounded between $s$ and $t$. Then we may consider $\lambda$ as a partition into exactly $t$ nonnegative parts, $\lambda = (\nklist{\ell}{t})$. Note that $\lambda$ has at least $s$ occurrences of its largest part, that is, $\ell_1 = \ell_2 = \dots = \ell_s$. 

To account for $q^{\tfrac{t^2+t}{2}}$, we add a staircase to $\lambda$. That is, we add $t$ to the first part, $t-1$ to the second part, and so on, adding 1 to the last part. At this stage, $\lambda$ contains the sequence $(\ell_1 +t, \ell_1+(t-1), \dots, \ell_1+(t-s+1))$, which implies that $\sigma(\lambda) \geq s$. Finally, since $\ell(\lambda) = m + t$ and $\#(\lambda) = t$, we see that $r(\lambda) = m$.
\end{proof}

We also need an overpartition statistic introduced by Corteel and Lovejoy \cite{Opartns} \cite{LovejoyD}. Given an overpartition $\lambda$, the \emph{overpartition rank} of $\lambda$ is defined to be
	\begin{align*}
		\oo{r}_{CL}(\lambda) := \ell(\lambda) - 1 - \#(\lambda_<),
	\end{align*}
where $\lambda_<$ is the suboverpartition whose parts are  all the overlined parts of $\lambda$ smaller than $\ell(\lambda)$.
Here we have chosen the notation $\oo{r}_{CL}(\lambda)$ in order to avoid confusion in the ranks.

For example, if $\lambda = (\overline{5},\overline{3},3,\overline{1})$, then $\lambda_< = (\oo{3}, \oo{1}) $, and $\oo{r}_{CL}(\lambda) = 5-1-2=2$.
Note that if every part of $\lambda$ is overlined, then $\oo{r}_D(\lambda) = \oo{r}_{CL}(\lambda)$.

We introduce a variant of the bracket for overpartitions.
If $\lambda = (\nklist{\ell}{n})$ is an overpartition, then the \emph{overpartition bracket} of $\lambda$ is defined to be the length of the longest sequence of the form $(\nklist{\ell}{k})$, where for all $1 \leq i <k$, we have one of the following:
\begin{itemize}
\item $\ell_i = \ell_{i+1}$
\item $\ell_i = \ell_{i+1}+1$ and at least  one of $\ell_i$ and $\ell_{i+1}$ is overlined.
\end{itemize}
We denote the overpartition bracket of $\lambda$ by $\overline{\sigma}(\lambda)$.

For example, if $\lambda = (7,7,\overline{6},5,4)$, then we consider the sequences
\begin{align*}
	\begin{matrix}
		(7), &
		(7,7), &
		(7,7,\overline{6}), &
		(7,7,\overline{6},5),
	\end{matrix}
\end{align*}
the longest of which has length four. Therefore, $\os(\lambda)=4$.

We see how the overpartition rank and the overpartition bracket relate to \eqref{gen1} in the following lemma.

\begin{lem} \label{overrun}
Fix nonnegative integers $1 \leq s\leq t$. The coefficient of $z^mq^n$ in
\begin{align*}
	\frac{\aqprod{-1}{q}{t}}{\aqprod{zq^s}{q}{t-s+1}}
\end{align*}
is equal to the number of overpartitions $\lambda$ of $n$ into $t$ nonnegative parts with $\os(\lambda) \geq s$ and $m = \oo{r}_{CL}(\lambda) + 1$.
\end{lem}

The proof of Lemma \ref{overrun} relies on an an algorithm originally due to Joichi and Stanton \cite{Joichi}.

\begin{alg}[Joichi, Stanton \cite{Joichi}]
\label{alg:JS}
Input: a partition $\lambda = (\nklist{\ell}{n})$ into $n$ parts, and a partition $\mu = (\nklist{m}{k})$ into $k$ distinct nonnegative parts, each less than $n$.

\noindent Output: An overpartition $\lambda' = (\ell'_1, \ell'_2, \dots, \ell_n')$ into $n$ parts.

\begin{enumerate}
\item Delete $m_1$ from $\mu$, and add 1 to the first $m_1$ parts of $\lambda$. This operation is well defined, as all parts of $\mu$ are strictly less than the number of parts of $\lambda$. Because $\mu$ is a partition into nonnegative parts, 0 may occur as a part of $\mu$.
If $m_1=0$, then the parts of $\lambda$ are unchanged.

\item Overline the $(m_1+1)$-st part of $\lambda$. If $m_1=0$, then we overline $\ell_1$.

\item Relabel the parts of $\mu$, if any exist, so that $m_1$ is the largest part of $\mu$. Repeat Steps (1) to (3) until the parts of $\mu$ are exhausted.
\end{enumerate}
\end{alg}

Because the parts of $\mu$ are distinct, we see that $\lambda'$ is an overpartition into $n$ parts.
An example of the Joichi Stanton map shown in Table  \ref{dissect}.
Algorithm \ref{alg:JS} is not difficult to reverse; additional details may be found in work of Lovejoy \cite{LovejoyD}. We now prove Lemma \ref{overrun}.

\begin{proof}[Proof of Lemma \ref{overrun}]
As in the proof of Lemma \ref{initrun}, the term 
\begin{align*}
\frac{1}{\aqprod{zq^s}{q}{t-s+1}}
\end{align*}
generates  a partition $\lambda$ into exactly $t$ nonnegative parts, with at least $s$ occurrences of its largest part, and with its largest part equal to $m$. The term $\aqprod{-1}{q}{t}$ generates a partition $\mu$ into distinct nonnegative parts less than $t$. We now apply Algorithm \ref{alg:JS} to produce an overpartition $\lambda'$. 
We claim that the overpartition bracket of $\lambda'$ is equal to the number of occurrences of the largest part of $\lambda$.

We induct on the number of parts of $\mu$. If $\mu =\emptyset$, then $\lambda'$ has no overlined parts, and $\overline{\sigma}(\lambda')$ is equal to the number of occurrences of the largest part of $\lambda'$, which is at least $s$.
 
Suppose that $\mu=(\nklist{m}{k+1})$ and let $\lambda'$ be overpartition corresponding to the pair $(\lambda, (\nklist{m}{k}))$. Let $\alpha = (\ell'_1, \ell'_2, \dots, \ell'_j)$ be the sequence which determines the overpartition bracket of $\lambda$. It is sufficient to show that Algorithm \ref{alg:JS} leaves the length of $\alpha$ unchanged. If $m_{k+1}<j$, then all parts of $\alpha$ are increased by 1. Thus $(\ell'_1 +1, \ell'_2 +1, \dots, \ell'_j +1)$ is eligible for determining $\os(\lambda)$, but neither of the sequences $(\ell'_1 +1, \ell'_2 +1, \dots, \ell'_j +1, \ell'_{j+1}+1)$ or $(\ell'_1 +1, \ell'_2 +1, \dots, \ell'_j +1, \overline{\ell'_{j+1}})$ are eligible. Therefore, the length of $\alpha$ is unchanged.

Otherwise, if $m_{k+1}\leq j$, then the sequence 
\begin{align*}
(\ell'_1 +1,  \dots, \ell'_{m_{k+1}-1}+1, \overline{\ell_{m_{k}+1}}, \ell_{m_{k}+2}, \dots, \ell'_j)
\end{align*}
  is eligible for determining $\os(\lambda)$, but 
\begin{align*}
(\ell'_1 +1,  \dots, \ell'_{m_{k+1}-1}+1, \overline{\ell_{m_{k}+1}}, \ell_{m_{k}+2}, \dots, \ell'_j, \ell'_{j+1})
\end{align*}
 is not.
Therefore, the length of $\alpha$ is unchanged.
That is, $\os(\lambda')$ is invariant under iterations of Algorithm \ref{alg:JS}.

Recall that $\ell(\lambda) = m$. Each iteration of Algorithm \ref{alg:JS} increases the largest part of $\lambda'$ by 1, except for the case $m_k=0$. Thus, the largest part of $\lambda'$ is equal to $m$ plus the number of overlined parts less than $\ell(\lambda')$.
Then 
\begin{align*}
	\oo{r}_{CL}(\lambda') = [\ell(\lambda) + \#(\lambda'_<)] - 1 - \#(\lambda'_<)
	= \ell(\lambda) -1 = m-1,
\end{align*}
as desired.
\end{proof}

\begin{table}[h]
\centering
\begin{tabular}{|c|l|l|l|l|}
\hline
Iteration & $\lambda$ & $\mu$ & $\os(\lambda)$ & $\oo{r}_{CL}(\lambda)$ \\
\hline
0&$(4,3,2,2)$&$(3,1,0)$ & $1$ & $3$ \\
1&$(5,4,3,\overline{2})$&$(1,0)$ & $1$ & $3$ \\
2&$(6,\overline{4},3,\overline{2})$&$(0)$ & $1$  & $3$ \\
3&$(\overline{6},\overline{4},3,\overline{2})$&$\emptyset$ & $1$ & $3$ \\
\hline
\end{tabular}
\caption{\label{dissect} An example of Algorithm \ref{alg:JS}.}
\end{table}

We can now give a combinatorial interpretation of \eqref{gen1} in terms of buffered Frobenius representations.


\begin{defn}
A buffered Frobenius representation of the first kind, or a $B_1$-representation for short, is a buffered Frobenius representation
\begin{align*}
\nu \in
\begin{pmatrix}
	\nkand{A}{k}\\
	\nkand{B}{k}
\end{pmatrix},
\end{align*}
in which
\begin{enumerate}
	\item $A_1$ is the set of nonempty partitions $\alpha_1$ into distinct parts.
	\item $A_2$ is the set of nonempty partitions $\alpha_2$ with  $\#(\alpha_2) \leq \sigma(\alpha_1)$.
	\item For all $i \geq 3$, the set $A_i$ is the set of nonempty partitions $\alpha_i$ with  $\#(\alpha_i)$ less than or equal to the number of occurrences of the largest part of $\alpha_{i-1}$.
	\item $B_1$ is the set of overpartitions $\beta_1$ into $\#(\alpha_1)$ nonnegative parts with $\oo{\sigma}(\beta_1) \geq \#(\alpha_2)$.
	\item For all $2 \leq i < k$, the set $B_i$ is the set of partitions into $\#(\alpha_i)$ nonnegative parts with at least $\#(\alpha_{i+1})$ occurrences of its largest part.
	\item $B_k$ is the set of partitions into $\#(\alpha_k)$ nonnegative parts. 
\end{enumerate}
We also define the empty array to be a $B_1$-representation with $k=0$.
\end{defn}
For example, consider the array:

\begin{align}
	\nu =
	\begin{pmatrix}
		\widehat{(3,2,1)} & (2,2,1)&(3)\\
		(\overline{4},4,\overline{3})&\widehat{(1,0,0)}&(0)
	\end{pmatrix}
\end{align}
On the top row, $\alpha_1$ is a partition into distinct parts, which satisfies (1).
Next, $\alpha_2$ is a partition into three parts with two occurrences of its largest part. Because $\sigma(\alpha_1)=3,$ this satisfies (2).
Finally, $\alpha_3$ is a nonempty partition with one part.
Because $\alpha_2$ has two occurrences of its largest part, this satisfies (3).

On the bottom row, $\beta_1$ is an overpartition into three parts with $\os(\beta_1)=3$, which satisfies (4).
Next, $\beta_2$ is a partition into three nonnegative parts, with one occurrence of its largest part, which satisfies (5).
Finally,  $\beta_3$ is a partition into one nonnegative part, which satisfies (6).
Additionally, both $\alpha_1$ and $\beta_2$ are marked with hats.

As in Section \ref{sec:buff}, we see that Lovejoy's first Frobenius representations of overpartitions correspond to the case $k = 1$ above.
For $k > 1$, we can collapse $B_1$-representations using the jigsaw map.
\begin{prop}
	Let $\mathcal{B}_1$ denote the set of $B_1$-representations, and let $\mathcal{F}_1$ denote the set of first Frobenius representations of overpartitions. Then $j:\mathcal{B}_1 \to \mathcal{F}_1$ is a surjective map. 
\end{prop}
Taken with Theorem \ref{Dbi}, we see that every $B_1$-representation $\nu$ corresponds to an overpartition $\lambda$, although this correspondence is many-to-one.
Thus the ranks we will establish to study $\Frobk$ do not immediately carry over to the set of overpartitions.

\subsection{Ranks of $B_1$-representations}

If
\begin{align*}
	\nu =
	\abarray,
\end{align*}
 then $\nu$ admits $k$ different rank functions, corresponding to the $x_i$ variables  in $\Frobk$.
We first define the indicator function $\chi_i$ to be
\begin{align*}
	\chi_{i}(\nu) := \begin{cases}
		1 &: \alpha_i \mbox{ is marked with a hat, and } \beta_i \mbox{ is not marked with a hat}\\
		-1 &: \beta_i \mbox{ is marked with a hat, and } \alpha_i \mbox{ is not marked with a hat}\\
		0 &: \text{ otherwise.}
	\end{cases}
\end{align*}
  We see that  $\chi_i$ detects buffers in the tableaux of $\nu$.
 The \emph{first rank} of $\nu$ is defined to be
\begin{align}
	\rho^1_1(\nu) := 
	 r(\alpha_1) - (\overline{r}_{CL}(\beta_1) +1) + \chi_1(\nu).
\end{align}
We also define $\rho^1_1(\emptyset) := 0$.

For $1 < i \leq k$, the \emph{$i$th rank} of $\nu$ is defined to be
\begin{align*}
	\rho_1^{i}(\nu) = (\ell(\alpha_i) - 1) - \ell(\beta_i) + \chi_i(\nu ).
\end{align*}

\noindent We also define $\rho_1^i(\nu) := 0$ whenever $\nu$ has fewer than $i$ columns.

For example, let
\begin{align*}
	\nu =
	\begin{pmatrix}
		\widehat{(3,2,1)} & (2,2,1)&(3)\\
		(\overline{4},4,\overline{3})&\widehat{(1,0,0)}&(0)
	\end{pmatrix}
\end{align*}
Then 
\begin{align*}
	\rho_1^1(\nu) &= (3-3) - ((3-1)+1) + 1 = -2\\
	\rho_1^2(\nu) &= (2 - 1) - 1 - 1 = -1\\
	\rho_1^3(\nu) &= (3-1) - 0 - 0 = 2,
\end{align*}
and  $\rho_1^i(\nu)=0$ for $i>3$.

We now establish $\Frobk$ as the generating series for the ranks of $B_1$-representations.

\subsection{Generating Series}


Let $\mathcal{B}_1^k$ denote the set of $B_1$-representations with at most $k$ columns,
\begin{align*}
	\mathcal{B}_1^k := 
	\left\{
	\begin{pmatrix}
		\nkand{\alpha}{j}\\
		\nkand{\beta}{j}
	\end{pmatrix}
	\in
	\mathcal{B}_1\
	\middle|\
	j \leq k
	\right\}
	.
\end{align*}

\begin{thm} \label{1summand}
The coefficient of $\nkprex{x}{k} q^n$ in
\begin{align*}
	\sum_{{
		n_1 ,
		\dots ,
		n_k \geq 0		
	}}
	\aqprod{-1}{q}{N_k}q^{\tfrac{N_k^2-N_k}{2}}
	\prod_{i=1}^{k} \frac{(1-x_{k-i+1})(1-{x^{-1}_{k-i+1}})q^{N_i}}{(x_{k-i+1}q^{N_{i-1}},x^{-1}_{k-i+1}q^{N_{i-1}})_{n_i+1}}
\end{align*}
is equal to the number of $B_1$-representations $\nu \in \mathcal{B}_1^k$ such that $|\nu|=n$ and $\rho_1^i(\nu)=m_i$, where the  count is weighted by $(-1)^{h(\nu)}$.
\end{thm}

\begin{proof}
Consider an arbitrary summand of the form
\begin{align*}
		\aqprod{-1}{q}{N_k}q^{\tfrac{N_k^2-N_k}{2}}
	\prod_{i=1}^{k} \frac{(1-x_{k-i+1})(1-{x^{-1}_{k-i+1}})q^{N_i}}{(x_{k-i+1}q^{N_{i-1}},x^{-1}_{k-i+1}q^{N_{i-1}})_{n_i+1}}.
\end{align*}
If $n_1 = \cdots = n_k = 0$, then the summand reduces to  $1$, which corresponds to the empty $B_1$-representation $\nu = \emptyset$.
Otherwise, $n_i > 0$ for some $i$.
Let $j$ be the smallest index so that $n_{j} > 0$.
Then the summand reduces to
\begin{align} \label{eq:B1-jsum}
	\aqprod{-1}{q}{N_k}q^{\tfrac{N_k^2-N_k}{2}}
	\prod_{i=j}^{k} \frac{(1-x_{k-i+1})(1-{x^{-1}_{k-i+1}})q^{N_i}}{(x_{k-i+1}q^{N_{i-1}},x^{-1}_{k-i+1}q^{N_{i-1}})_{n_i+1}}.
\end{align}
We claim that the coefficient of $\nkprex{x}{k}q^n$ in \eqref{eq:B1-jsum} is equal to the number of $B_1$-representations 
\begin{align*}
	\nu =
	\begin{pmatrix}
		\nkand{\alpha}{k-j+1}\\
		\nkand{\beta}{k-j+1}
	\end{pmatrix}
\end{align*}
where $\#(\alpha_i) = N_{k-i+1}$, such that $|\nu|=n$ and $\rho_1^i(\nu)=m_i$, where the  count is weighted by $(-1)^{h(\nu)}$.
Note that 
\begin{align*}
	  (-1)^{h(\nu)} = (-1)^{\sum \chi_i (\nu)}.
\end{align*}

The parts of $\alpha_1$ and $\beta_1$ are generated by the $i=k$ multiplicand, which we write as
\begin{align} \label{eq:B1-1}
	\bigg(
		\frac{(1-x_1) q^{\tfrac{N_k^2+N_k}{2}} }{\aqprod{x_1 q^{N_{k-1}}}{q}{n_k+1}}
	\bigg)
	\bigg(
		\frac{(1-{x^{-1}_{1}}) \aqprod{-1}{q}{N_k}}{\aqprod{x^{-1}_{1}q^{N_{k-1}}}{q}{n_k+1}}
	\bigg).
\end{align}

We use the fact that $N_{k} = N_{k-1} + n_{k}$ to apply Lemmas \ref{initrun} and \ref{overrun} with $t=N_k$ and $s=N_{k-1}$. Then we see that $\alpha_1$ is a partition into distinct parts with $\sigma(\alpha_1) \geq N_{k-1}$ and $\beta_1$ is an overpartition into $N_k$ nonnegative parts with $\os(\beta_1)\geq N_{k-1}$.
Here, the exponents of $x_1$ and $x_1^{-1}$ track $r(\alpha_1)$ and $\oo{r}_{CL}(\beta_1)+1$, respectively.

Given an arbitrary  $(\alpha_1, \beta_1)$, the coefficient of $x_1^{m_1}$ in 
$
	(1-x_1)(1-x_1^{-1}) 
$
 is equal to the weighted count of ways to mark $\alpha_1$ or $\beta_1$ with hats, where $m_1 = \chi_1(\nu)$ and the count is weighted by $(-1)^{\chi_1(\nu)}$.
Therefore, the coefficient of $x_1^{m_1} q^n$ in \eqref{eq:B1-1} is equal to the weighted count of possible columns $(\alpha_1, \beta_1)^T$ of a $B_1$-representation $\nu$ such that $\# (\alpha_1) = N_k$, $\#(\alpha_2) = N_{k-1}$, $n = |\alpha_1| + |\beta_1|$, and  
\begin{align*}
	m_1 = r(\alpha_1) - (\oo{r}_{CL}(\beta_1)+1) + \chi_1(\nu) = \rho_1^1 (\nu),
\end{align*}
 where the count is weighted by $(-1)^{\chi_1(\nu)}$.

For $1<i<k-j+1$, the parts of $\alpha_i$ and $\beta_i$ are generated by the $k-i+1$ multiplicand, which we write as
\begin{align} \label{eq:B1-mid}
	\bigg(
		\frac{(1-x_{i})q^{N_{k-i+1}}}{\aqprod{x_{i}q^{N_{k-i}}}{q}{n_{k-i+1}+1}}
	\bigg)
	\bigg(
		\frac{1-{x^{-1}_{i}}}{\aqprod{x^{-1}_{i}q^{N_{k-i}}}{q}{n_{k-i+1}+1}}
	\bigg).
\end{align}
As in Lemma \ref{initrun},
\begin{align*}
	\frac{q^{N_{k-i+1}}}{\aqprod{x_{i}q^{N_{k-i}}}{q}{n_{k-i+1}+1}}
\end{align*}
 generates the Young tableau of $\alpha_i$,  whose columns' lengths are bounded between $N_{k-i}$ and ${N_{k-i+1}}$.
We add $1$ to each part of $\alpha_i$ to account for $q^{N_{k-i+1}}$.
Thus, $\alpha_i$ is a nonempty partition with $N_{k-i+1}$ positive parts and at least $N_{k-i}$ occurrences of its largest part, and $\beta_{i}$ is a partition into $N_{k-i+1}$ nonnegative parts with at least $N_{k-i}$ occurrences of its largest part.
Here, the exponents of $x_i$ and $x_i^{-1}$ track $\ell(\alpha_1)-1$ and $\ell(\beta_1)$, respectively.

Because $(\alpha_i, \beta_i)^T$ is not the rightmost column of $\nu$, either entry may be marked with a hat. As with the previous column, entries marked by a hat are tracked by the term $(1-x_i)(1-x_i^{-1})$.
Therefore, the coefficient of $x_i^{m_i} q^n$ in \eqref{eq:B1-mid} is equal to the weighted count of possible columns $(\alpha_i, \beta_i)^T$ of $\nu$ such that $\# (\alpha_i) = N_{k-i+1}$, $\# (\alpha_{i+1}) = N_{k-i}$, $n = |\alpha_i| + |\beta_i|$, and  
\begin{align*}
	m_i = (\ell(\alpha_i) - 1) - \ell(\beta_i) + \chi_i(\nu) = \rho_1^i (\nu),
\end{align*}
 where the count is weighted by $(-1)^{\chi_i(\nu)}$.

 Finally, the parts of $\alpha_{k-j+1}$ and $\beta_{k-j+1}$ are generated by the $i= k-j+1$ multiplicand,
\begin{align*}
	\bigg(
		\frac{(1-x_{k-j+1})q^{N_{j}}}{\aqprod{x_{k-j+1}q^{N_{j-1}}}{q}{n_{j}+1}}
	\bigg)
	\bigg(
		\frac{(1-{x^{-1}_{j}})}{\aqprod{x^{-1}_{k-j+1}q^{N_{j-1}}}{q}{n_{j}+1}}
	\bigg).
\end{align*}
By minimality of $j$, we see that $n_1 = \cdots = n_{j-1} = 0$. Thus, $N_{j-1} = 0$, and the multiplicand reduces to
\begin{align} \label{eq:B1-last}
	\bigg(
		\frac{q^{N_j}}{\aqprod{x_{k-j+1}q}{q}{n_j}}
	\bigg)
	\bigg(
		\frac{1}{\aqprod{x^{-1}_{k-j+1}}{q}{n_j}}
	\bigg).
\end{align}
This reflects the fact that neither $\alpha_{k-j+1}$ or $\beta_{k-j+1}$ can be marked with a hat.
As with the previous column, we see that the coefficient of $x_{k-j+1}^{m_{k-j+1}} q^n$ in \eqref{eq:B1-last} is equal to the weighted count of possible columns $(\alpha_{k-j+1}, \beta_{k-j+1})^T$ of $\nu$
 such that $\# (\alpha_{k-j+1}) = N_{k-i+1}$,
 $n = |\alpha_{k-j+1}| + |\beta_{k-j+1}|$,
 and  $m_{k-j+1} =  \rho_1^{k-j+1} (\nu)$,
 where the count is weighted by $(-1)^{\chi_{k-j+1}(\nu)}$.

By combining these terms, we have counted all possible $\nu \in \mathcal{B}_1^k$ 
 with $|\alpha_i| = N_{k-i+1}$, $|\nu| = n$, $\rho_1^i(\nu) = m_i$, and $h(\nu)$ entries marked with a hat, where the count is weighted by $(-1)^{h(\nu)}$.
By summing over all values of $\nklist{n}{k}$, we generate all possible $B_1$-representations in $\mathcal{B}_1^k$.

\end{proof}

\subsection{Full Rank and Proof of Theorem \ref{thm1}}

We have one final statistic in this section. We define the \emph{full rank} of a $B_1$-representation $\nu$ to be the sum of the $i$th ranks of $\nu$, 

\begin{align*}
	\rho_1(\nu) := \sum_{i \geq 1} \rho_1^i(\nu).
\end{align*}

\noindent This sum converges for any $B_1$-representation $\nu$, as all but finitely many of the summands vanish. We may now prove Theorem \ref{thm1}.

\begin{proof}[Proof of Theorem \ref{thm1}]
Let $\zeta_k$ be a primitive $k$th root of unity. The desired generating series, 
\begin{align*}
	\sum_{\nu \in \mathcal{B}^k_1} (-1)^{h(\nu)} \prod_{i=1}^k\zeta_k^{ (i-1)\rho_1^i(\nu)} z^{\tfrac{\rho_1(\nu)}{k}} q^{|\nu|},
\end{align*}
is given by 
\begin{align} \label{punchline1}
\overline{R}_k(\sqrt[k]{z},\zeta_k\sqrt[k]{z},\dots, \zeta_k^{k-1}\sqrt[k]{z};q) = \Rk.
\end{align}
\end{proof}

We now have our combinatorial interpretation of $\Rk$.
 Observe that one of the series in \eqref{punchline1} is a series in $\sqrt[k]{z}$ with coefficients in $\ZZ[\zeta_k]$, and the other is a series in $z$ with integer coefficients. Thus, the weighted count must vanish for $B_1$-representations whose full rank is not a multiple of $k$.

We close this section by discussing conjugation maps on $\mathcal{B}_1^k$.

\subsection{Conjugation}

Given a buffered Frobenius representation of the first kind
\begin{align*}
\nu = \abarray,
\end{align*}
we define $k$ different conjugation maps corresponding to the columns of $\nu$. To perform the \emph{first conjugation}, delete a staircase from $\alpha_1$ by removing $\#(\alpha_1)$ from the first part, $\#(\alpha_1)-1$ from the second part and so on until removing $1$ from the smallest part. We next reverse Algorithm \ref{alg:JS} on $\beta_1$. Let $\lambda$ and $\mu$ be the partition and partition into distinct parts produced this way, respectively. Both $\alpha_1$ and $\lambda$ are partitions into $\#(\alpha_1)$ nonnegative parts with at least $\#(\alpha_2)$ occurrences of their largest parts. Add a staircase to $\lambda$ to produce $\alpha_1'$, and perform  Algorithm \ref{alg:JS} on $\alpha_1$ and $\mu$ to produce $\beta_1'$. 
We mark $\alpha_1'$ with a hat if and only if $\beta_1$ was marked with a hat, and vice versa.
We call
\begin{align*}
	\phi_1^1(\nu) :=
	\begin{pmatrix}
		\alpha'_1 & \alpha_2 & \dots & \alpha_k\\
		\beta'_1 & \beta_2 & \dots & \beta_k
	\end{pmatrix}
\end{align*}
the \emph{first conjugate} of $\nu$.

For example, let
\begin{align*}
	\nu =
	\begin{pmatrix}
		\widehat{(3,2,1)} & (2,2,1)&(3)\\
		(\overline{4},4,\overline{3})&\widehat{(1,0,0)}&(0)
	\end{pmatrix}.
\end{align*}
Then removing the staircase from $\alpha_1$ produces
\begin{align*}
	\alpha_1 = (0, 0, 0),
\end{align*}
while reversing Algorithm \ref{alg:JS} on $\beta_1$ produces
\begin{align*}
	\lambda &= (3, 3, 3)\\
	\mu &= (2, 0).
\end{align*}
Next, we add a staircase to $\lambda$, and perform Algorithm \ref{alg:JS} on $\alpha_1$ and $\mu$, producing
\begin{align*}
	\alpha_1' &= (6, 5, 4)\\
	\beta_1' &= (\oo{1}, 1, \oo{0}).
\end{align*}
Because $\alpha_1$ was marked with a hat, and $\beta_1$ was not marked with a hat, we see that
\begin{align*}
	\phi_1^1(\nu) =
	\begin{pmatrix}
		(6, 5, 4) & (2,2,1)&(3)\\
		\widehat{(\oo{1}, 1, \oo{0})} &\widehat{(1,0,0)}&(0)
	\end{pmatrix}.
\end{align*}

For $i > 1 $, the \emph{$i$th conjugation map} is performed as follows. First, subtract 1 from each part of $\alpha_i$ to produce $\beta_i'$, and add 1 to each part of $\beta_i$ to produce $\alpha_i'$. 
We mark $\alpha_i'$ with a hat if and only if $\beta_i$ was marked with a hat, and vice versa.
We call
\begin{align*}
	\phi_1^i(\nu) :=
	\begin{pmatrix}
	\alpha_1 &  \dots & \alpha_i' & \dots & \alpha_k\\
	\beta_1 &  \dots & \beta_i' & \dots & \beta_k
	\end{pmatrix}
\end{align*}
the \emph{$i$th conjugate} of $\alpha$.
We also define $\phi_1^i(\nu) := \nu$ if $\nu$ has fewer than $i$ columns.

For example, we see that 
\begin{align*}
	\phi_1^2 (\nu) &=
	\begin{pmatrix}
		\widehat{(3,2,1)} & \widehat{(2,1,1)}  &(3)\\
		(\overline{4},4,\overline{3}) & (1,1,0) & (0)
	\end{pmatrix}\\
	\phi_1^3(\nu) &=
	\begin{pmatrix}
		\widehat{(3,2,1)} & (2,2,1)&(1)\\
		(\overline{4},4,\overline{3})&\widehat{(1,0,0)}&(2)
	\end{pmatrix}.
\end{align*}

Each of the $i$th conjugation maps exchange the roles of
\begin{align*}
\frac{1-x_i}{\aqprod{x_i q^{N_{k-i}}}{q}{n_{k-i+1}+1}}
\mbox{ and }
\frac{1-x_i^{-1}}{\aqprod{x_i^{-1} q^{N_{k-i}}}{q}{n_{k-i+1}+1}}
\end{align*}
in \eqref{gen1}. This fact immediately implies two propositions. 

\begin{prop}
For all $i\geq 1$, we have $\rho_1^i(\phi_1^i(\nu)) = -\rho_1^i(\nu)$.
\end{prop}

\begin{prop}
For all nonnegative integers $i$ and $j$, $\phi_1^i\phi_1^j=\phi_1^j\phi_1^i$.
\end{prop}
Finally, if we define the \emph{full conjugation} to be
\begin{align*}
	\phi_1 := \prod_{i \geq 1} \phi_1^i,
\end{align*}
then $\phi_1$ is defined for all $\nu \in \mathcal{B}_1$, and $\rho_1(\phi_1(\nu)) = -\rho_1(\nu)$.

We now consider a second family of buffered Frobenius representations.

\section{Buffered Frobenius Representations of the Second Kind}
\label{second}

Recall that  $\overline{R[2]}(z;q)$ is the generating series for the $M_2$-rank of  overpartitions.
We consider the series
\begin{multline*}
	\FrobkII : = \ovr\\
	\times \left(1+2\sum_{n=1}^\infty (-1)^{n} q^{n^2+2kn}\prod_{i=1}^k \frac{(1-x_i)(1-x_i^{-1})}{(1-x_iq^{2n})(1-x_i^{-1}q^{2n})}\right),
\end{multline*}
bearing in mind that
$$
	\overline{R2}_{k}(\sqrt[k]{z},\zeta_k\sqrt[k]{z},\dots, \zeta_k^{k-1}\sqrt[k]{z};q) = \overline{R[2k]}(z,q).
$$
The thoughtful reader may be concerned that we are reproducing the work of Section \ref{first}.
We will see that buffered Frobenius representations of the second kind  directly generalize Lovejoy's second Frobenius representation of overpartitions, as opposed to the multi-to-one correspondence that $B_1$-representations require.
We hope that studying both of these families will allow us to define an infinite family of overpartition ranks, as we discuss in Section \ref{end}.

We see a transformation of $\FrobkII$ in the theorem below.

\begin{thm} \label{2ndhype} Let $k \geq 1$ be a positive integer. Then we have
\begin{multline*}
	\ovr \left(1+2\sum_{n=1}^\infty (-1)^n q^{n^2+2kn}\prod_{i=1}^k \frac{(1-x_i)(1-x_i^{-1})}{(1-x_iq^{2n})(1-x_i^{-1}q^{2n})}\right)\\
	=
		\sum_{{
			n_1 ,
			\dots ,
			n_{k} \geq 0
		}}
	\frac{\aqprod{-1}{q}{2N_k}}{q^{N_k}} \prod_{i=1}^{k} \frac{(1-x_{k-i+1})(1-x_{k-i+1}^{-1})q^{2N_i}}{\aqprod{x_{k-i+1}q^{2N_{i-1}},x_{k-i+1}^{-1}q^{2N_{i-1}}}{q^2}{n_i+1}},
\end{multline*}
where we write $N_0 := 0$ and for all $1\leq i \leq k$, we write $N_i = \nkplus{n}{i}$.
\end{thm}

\begin{proof}
We begin by substituting $k \mapsto (k+1)$ and $q \mapsto q^2$ in Corollary \ref{swerdnA}.
Then we have
\begin{multline*}
	\pPq{2k+6}{2k+5\!\!}{a, q^2a^{\tfrac{1}{2}}, -q^2a^{\tfrac{1}{2}}, b_1, c_1, \dots, b_{k+1}, c_{k+1}, q^{-2N}}{a^{\tfrac{1}{2}}, -a^{\tfrac{1}{2}}, \tfrac{aq}{b_1}, \tfrac{aq}{c_1},  \dots, \tfrac{aq}{b_{k+1}}, \tfrac{aq}{c_{k+1}}, aq^{2N+2} }{q^2;  \frac{a^kq^{2k+2N}}{\prod_{i=1}^{k+1} b_i c_i} } \\
	=\frac{\aqprod{aq^2,\tfrac{aq^2}{b_{k+1} c_{k+1}}}{q^2}{N}}{\aqprod{\tfrac{aq^2}{b_{k+1}},\tfrac{aq^2}{c_{k+1}}}{q^2}{N}}
		\sum_{{
			n_1,
			\dots ,
			n_{k} \geq 0
		}}
	\frac{\aqprod{\tfrac{aq^2}{b_kc_k}}{q^2}{n_1}}{\aqprod{q^2}{q^2}{n_1}}
	 \frac{\aqprod{\tfrac{aq^2}{b_{k-1}c_{k-1}}}{q^2}{n_2}}{\aqprod{q^2}{q^2}{n_2}} \cdots \frac{\aqprod{\tfrac{aq^2}{b_1c_1}}{q^2}{n_k}}{\aqprod{q^2}{q^2}{n_k}}\\
	\times \frac{\aqprod{b_{k-1}, c_{k-1}}{q^2}{N_1}}{\aqprod{\tfrac{aq^2}{b_k}, \tfrac{aq^2}{c_k}}{q^2}{N_1}} \frac{\aqprod{b_{k-2}, c_{k-2}}{q^2}{N_2}}{\aqprod{\tfrac{aq^2}{b_{k-1}}, \tfrac{aq^2}{c_{k-1}}}{q^2}{N_2}} \cdots  \frac{\aqprod{b_{1}, c_{1}}{q^2}{N_{k-1}}}{\aqprod{\tfrac{aq^2}{b_2}, \tfrac{aq^2}{c_2}}{q^2}{N_{k-1}}}  \\
\times \frac{\aqprod{b_{k+1}, c_{k+1}}{q^2}{N_k}}{\aqprod{\tfrac{aq^2}{b_1}, \tfrac{aq^2}{c_1}}{q^2}{N_k}} \frac{\aqprod{q^{-2N}}{q^2}{N_k}}{\aqprod{a^{-1}b_{k+1}c_{k+1}q^{-2N}}{q^2}{N_k}}\\
	\times \frac{(aq^2)^{\nkplus{N}{k-1}}q^{2N_k}}{(b_{k}c_{k})^{n_1} (b_{k-1}c_{k-1})^{N_2}\cdots (b_{1}c_{1})^{N_{k-1}}}.
\end{multline*}

Next, we take the limit as $N \to \infty$ and set $a =1$. As in the proof of Theorem \ref{firsthype}, we use \eqref{pochcancel} and \eqref{Nlim} to simplify the $q$-Pochhammer symbols. The equation becomes
\begin{multline*}
	1+\sum_{n=1}^\infty
	(1+q^{2n})
	\frac{\aqprod{b_1, c_1, \ldots, b_{k+1}, c_{k+1}}{q^2}{n}(-1)^n q^{n^2-n}}{\aqprod{\tfrac{q^2}{b_1}, \tfrac{q^2}{c_1}, \dots, \tfrac{q^2}{b_{k+1}}, \tfrac{q^2}{c_{k+1}}}{q^2}{n}}
	\left( \frac{q^{2k+2}}{\prod_{i=1}^{k+1} b_i c_i} \right)^n\\
	=\frac{\aqprod{q^2,\tfrac{q^2}{b_{k+1}c_{k+1}}}{q^2}{\infty}}{\aqprod{\tfrac{q^2}{b_{k+1}},\tfrac{q^2}{c_{k+1}}}{q^2}{\infty}}
		\sum_{{
			n_1 ,
			\dots ,
			n_{k} \geq 0
		}}
	\frac{\aqprod{\tfrac{q^2}{b_kc_k}}{q^2}{n_1}}{\aqprod{q^2}{q^2}{n_1}}
	 \frac{\aqprod{\tfrac{q^2}{b_{k-1}c_{k-1}}}{q^2}{n_2}}{\aqprod{q^2}{q^2}{n_2}} \cdots \frac{\aqprod{\tfrac{q^2}{b_1c_1}}{q^2}{n_k}}{\aqprod{q^2}{q^2}{n_k}}\\
	\times \frac{\aqprod{b_{k-1},c_{k-1}}{q^2}{N_1}}{\aqprod{\tfrac{q^2}{b_k},\tfrac{q^2}{c_k}}{q^2}{N_1}} \frac{\aqprod{b_{k-2},c_{k-2}}{q^2}{N_2}}{\aqprod{\tfrac{q^2}{b_{k-1}},\tfrac{q^2}{c_{k-1}}}{q^2}{N_2}}\cdots \frac{\aqprod{b_{1},c_{1}}{q^2}{N_{k-1}}}{\aqprod{\tfrac{q^2}{b_2},\tfrac{q^2}{c_2}}{q^2}{N_{k-1}}} \frac{\aqprod{b_{k+1},c_{k+1}}{q^2}{N_k}}{\aqprod{\tfrac{q^2}{b_1},\tfrac{q^2}{c_1}}{q^2}{N_k}}\\
	\times 
	\frac{q^{2N_1 + 2N_2 + \cdots + 2N_k}}{(b_{k+1} c_{k+1})^{N_k}(b_{k}c_{k})^{N_1} (b_{k-1}c_{k-1})^{N_2} \cdots (b_{1}c_{1})^{N_{k-1}}}.
\end{multline*}

Continue, setting $b_i = x_i$ and $c_i = x_i^{-1}$ for $1 \leq i \leq k$. We now diverge from the proof of Theorem \ref{firsthype} by setting $b_{k+1} = -1$ and $c_{k+1} = -q$. The term
\begin{align*}
	(-1)^n q^{n^2 - n}
	\frac{\aqprod{c_{k+1}}{q^2}{n}}{\aqprod{\tfrac{q^2}{c_{k+1}}}{q^2}{n}}
	\left(
		\frac{q^{2k+2}}{b_{k+1} c_{k+1}}
	\right)^n
\end{align*}
in the left hand side of the equation
reduces to $(-1)^n q^{n^2 + 2kn}$,
 and we obtain
\begin{align*}
	1+\sum_{n=1}^\infty
	(1+q^{2n})
	\frac{\aqprod{x_1, x_1^{-1}, x_2, x_2^{-1}, \dots, x_k, x_k^{-1}, -1}{q^2}{n} (-1)^n q^{n^2+2kn}}{\aqprod{x_1q^2, x_1^{-1}q^2, x_2q^2, x_2^{-1}q^2, \dots, x_k q^2, x_k^{-1}q^2, -q^2}{q^2}{n}}.
\end{align*}
The right hand side of the equation becomes
\begin{multline*}
	\frac{\aqprod{q^2,q}{q^2}{\infty}}{\aqprod{-q^2,-q}{q^2}{\infty}}
		\sum_{{
			n_1,
			\dots ,
			n_{k} \geq 0
		}}
	\frac{\aqprod{x_{k-1},x^{-1}_{k-1}}{q^2}{N_1}}{\aqprod{x_kq^2,x_k^{-1}q^2}{q^2}{N_1}}\\
	\times \frac{\aqprod{x_{k-2},x^{-1}_{k-2}}{q^2}{N_2}}{\aqprod{x_{k-1}q^2,x_{k-1}^{-1}q^2}{q^2}{N_2}}\cdots \frac{\aqprod{x_{1},x^{-1}_{1}}{q^2}{N_{k-1}}}{\aqprod{x_2q^2,x_2^{-1}q^2}{q^2}{N_{k-1}}} \\
	\times \frac{\aqprod{-1,-q}{q^2}{N_k}q^{2N_1 + 2N_2 + \cdots + 2N_{k-1} + N_k}}{\aqprod{x_1q^2,x_1^{-1}q^2}{q^2}{N_k}} .
\end{multline*}

\noindent On the left hand side of the equation, we use Lemma \ref{reduce} to obtain
%
%
\begin{align*}
	1+2\sum_{n=1}^\infty (-1)^n q^{n^2+2kn}\prod_{i=1}^k \frac{(1-x_i)(1-x_i^{-1})}{(1-x_iq^{2n})(1-x_i^{-1}q^{2n})}.
\end{align*}

\noindent On the right hand side of the equation, we use Lemma \ref{reduce} and the relations
\begin{align*}
\frac{\aqprod{q^2,q}{q^2}{\infty}}{\aqprod{-q^2,-q}{q^2}{\infty}} &= \frac{\aqprod{q}{q}{\infty}}{\aqprod{-q}{q}{\infty}},\\
\aqprod{-1,-q}{q^2}{n} &= \aqprod{-1}{q}{2n}
\end{align*}
\noindent to obtain
\begin{multline*}
	\frac{\aqprod{q}{q}{\infty}}{\aqprod{-q}{q}{\infty}}
		\sum_{{
			n_1,
			\dots ,
			n_{k} \geq 0
		}}
	 \frac{\aqprod{-1}{q^2}{2N_k} q^{\nkplus{2N}{k-1}+N_k}}{\aqprod{x_kq^2, x_k^{-1}q^2}{q^2}{N_1}}\\
	\times \frac{(1-x_{k-1})(1-x_{k-1}^{-1})}{\aqprod{x_{k-1}q^{2N_1}, x_{k-1}^{-1}q^{2N_1}}{q^2}{n_2+1}}\\
	 \times \frac{(1-x_{k-2})(1-x_{k-2}^{-1})}{\aqprod{x_{k-2}q^{2N_2}, x_{k-2}^{-1}q^{2N_2}}{q^2}{n_3+1}}
	 \cdots \frac{(1-x_1)(1-x_1^{-1})}{\aqprod{x_{1}q^{2N_{k-1}}, x_{1}^{-1}q^{2N_{k-1}}}{q^2}{n_{k}+1}}.
\end{multline*}
Since $N_0 :=0$, the right side becomes
\begin{align*}
	\frac{\aqprod{q}{q}{\infty}}{\aqprod{-q}{q}{\infty}}
		\sum_{{
			n_1,
			\dots ,
			n_{k} \geq 0
		}}
	\frac{\aqprod{-1}{q^2}{2N_k}}{q^{N_k}}\prod_{i=1}^k \frac{(1-x_{k-i+1})(1-x_{k-i+1}^{-1})q^{2N_i}}{\aqprod{x_{k-i+1}q^{2N_{i-1}},x_{k-i+1}^{-1}q^{2N_{i-1}}}{q^2}{n_i+1}}.
\end{align*}
Here we have rewritten $q^{N_k}$ as $q^{2N_k}/q^{N_k}$ in order to simplify the product notation. Multiplying both sides by $\tovr$ gives us the desired equation,
\begin{multline} \label{2theorem}
	\ovr
	\left(1+
	2\sum_{n=1}^\infty 
	(-1)^n q^{n^2+2kn}\prod_{i=1}^k \frac{(1-x_i)(1-x_i^{-1})}{(1-x_iq^{2n})(1-x_i^{-1}q^{2n})}\right)\\
	= 
		\sum_{{
			n_1,
			\dots ,
			n_{k} \geq 0
		}}
	\frac{\aqprod{-1}{q}{2N_k}}{q^{N_k}} \prod_{i=1}^{k} \frac{(1-x_{k-i+1})(1-x_{k-i+1}^{-1})q^{2N_i}}{\aqprod{x_{k-i+1}q^{2N_{i-1}},x_{k-i+1}^{-1}q^{2N_{i-1}}}{q^2}{n_i+1}}.
\end{multline}
\end{proof}

\subsection{Overpartition Statistics}

In order to interpret \eqref{2theorem} as a generating series, we must introduce additional partition and overpartition statistics.
The first is a variation of Berkovich and Garvan's $M_2$-rank for partitions \cite{Berk} implied by work of Lovejoy \cite{LovejoyM2}.
Given a partition $\lambda$ into nonnegative parts where odd parts may not repeat, the $\emph{second partition rank}$ of $\lambda$ is defined to be
\begin{align*}
	r_2(\lambda) := \lfloor \tfrac{\ell(\lambda)}{2} \rfloor - \#(\lambda_{o,<}),
\end{align*}
where $\lambda_{o,<}$ is the subpartition of $\lambda$ consisting of all odd parts of $\lambda$ which are less than $\ell(\lambda)$.
For example, if $\lambda = (6,5)$, then $r_2(\lambda) = 3-1=2$.

We introduce another variation of the partition bracket.
Let $\lambda = (\nklist{\ell}{n})$ be a partition into nonnegative parts where odd parts may not repeat. The \emph{second bracket} of $\lambda$ is  the length of the longest substring of $\lambda$ of the form $(\nklist{\ell}{k})$, where for all $1\leq i < k$, we have $|\ell_{i+1}-\ell_i|<2$. We denote the second bracket of $\lambda$ by $\sigma_2(\lambda)$.

For example, if $\lambda = (8,8,7,6,4,4,2)$, then we consider the substrings
\begin{align*}
\begin{matrix}
(8),&
(8,8), &
(8,8,7), &
(8,8,7,6),
\end{matrix}
\end{align*}
the longest of which has length $4$.
Therefore, $\sigma_2(\lambda) =4$. 
We see how the second  rank and the second bracket relate to \eqref{2theorem} in the following lemma.

\begin{lem} \label{2frob1lemma}
Fix nonnegative integers $1 \leq s \leq t$. The coefficient of $z^mq^n$ in
\begin{align*}
\frac{\aqprod{-q}{q^2}{t}}{\aqprod{zq^{2s}}{q^2}{t-s+1}}
\end{align*}
is equal to the number of partitions $\lambda$ of $n$ into $t$ nonnegative parts where odd parts may not repeat with $r_2(\lambda) = m$ and  $\sigma_2(\lambda) \geq s$.
\end{lem}

The proof rests on Lovejoy's modification of Algorithm \ref{alg:JS}.

\begin{alg}[Lovejoy \cite{LovejoyM2}] \label{alg:JS2}
Input: A partition into $n$ nonnegative even parts $\lambda = (\nklist{\ell}{n})$, and a partition $\mu = (\nklist{m}{k})$ into $k$ distinct odd parts less than $2n$.

\noindent Output: A partition $\lambda' = (\ell'_1, \ell'_2, \dots, \ell_n')$ into $n$ nonnegative parts with $k$ distinct odd parts.

\begin{enumerate}
\item Delete the largest part of $\mu$, which we may write as $m_1 = 2s+1$.

\item Add 2 to the first $s$ parts of $\lambda$, then add 1 to $\ell_{s+1}$. Note that $\lambda_{s+1}$ is now odd. If $s=0$, then we instead add $1$ to $\lambda_1$. This operation is well defined, as $\lambda$ has exactly $n$ parts and $m_1 = 2s+1<2n$, which implies $s+1\leq n$.

\item Relabel the parts of $\mu$, if any exist, so that the largest part of $\mu$ is $m_1$. We now repeat Steps (1) and (2) until the parts of $\mu$ are exhausted.
\end{enumerate}
\end{alg}

Because the parts of $\mu$ are distinct, we see that $\lambda$ is a partition into $n$ nonnegative parts with $k$ distinct odd parts.

\begin{proof}[Proof of Lemma \ref{2frob1lemma}]

The term $$\frac{1}{\aqprod{zq^{2s}}{q^2}{t-s+1}}$$ generates pairs of columns in the Young tableau of a partition $\lambda$. Therefore, $\lambda$ has $t$ even nonnegative parts with at least $s$ occurrences of the largest part, and the coefficient of $z$ tracks one half of the largest part of $\lambda$.
The term $\aqprod{-q}{q^2}{t}$ generates a partition $\mu$ into distinct odd parts less than $2t$. 
We use Algorithm \ref{alg:JS2} to produce a partition $\lambda'$ into $t$ even nonnegative parts where odd parts may not repeat.
We claim that the second bracket of $\lambda'$ is equal to the number of occurrences of the largest part of $\lambda$, which is at least $s$.

To show that  $\sigma_2 (\lambda ') \geq s$, we induct on the number of parts of $\mu$. If $\mu$ is empty, then $\lambda'$ only consists of even parts.
In this case, $\sigma_2(\lambda')$ is equal to the number of occurrences of the largest part of $\lambda'$, which is $s$, and the second rank is equal to $m$.

Suppose that $\mu = (\nklist{m}{k+1})$ with $k+1$ parts, and let $\lambda'$ be the partition corresponding to $(\lambda, (\nklist{m}{k}))$. By assumption,  $\sigma_2(\lambda') \geq s$. Write $m_{k+1} = 2s_{k+1}+1$ and $m_k = 2s_k+1$. Because $m_{k+1}<m_k$, the first $s_{k+1}$ parts of $\lambda'$ must have the same parity.

If $\sigma_2(\lambda')\leq s_{k+1}$, then adding 2 to the first $s_{k+1}$ parts of $\lambda'$ will leave the second bracket unchanged. Otherwise, $\sigma_2(\lambda') > s_{k+1}$. In this case, adding 2 to the first $s_{k+1}$ parts of $\lambda'$ and adding $1 $ to $\ell_{s_{k+1}+1}$ also leaves the second bracket unchanged. In either case, we have shown that the result holds for a $\mu$ with $k+1$ parts.
Therefore, $\lambda$ is a partition of $n$ into $t$ nonnegative parts with $\sigma_2(\lambda)\geq s$.

Each step in Algorithm \ref{alg:JS2} adds an odd part to $\lambda$ and increases the largest part by either $1$ or $2$.
Let $\lambda'_o$ denote the subpartition whose parts are the odd parts of $\lambda'$ which are less than $\ell(\lambda)$.
Then $\ell(\lambda') = 2m + 2(\lambda_o)$ if $\ell(\lambda')$ is even, and $\ell(\lambda') = 2m + 2(\lambda_o) + 1$ if $\ell(\lambda')$ is odd.
In either case, we see that
$m = \lfloor \tfrac{\ell(\lambda)}{2} \rfloor - \#(\lambda'_o) = r_2(\lambda')$.
\end{proof}

We need a variation of the overpartition rank implied by the work of Lovejoy \cite{LovejoyM2}.
Given an overpartition $\lambda$ into odd parts, the \emph{second overpartition rank} of $\lambda$ is defined to be
\begin{align*}
	\overline{r}_2(\lambda) :=
	\tfrac{\ell(\lambda)-1}{2} - \#(\lambda_{<}),
\end{align*}
where we recall $\lambda_{<}$ is the sub-overpartition of $\lambda$ consisting of all overlined parts of $\lambda$ less than $\ell(\lambda)$.
For example, if $\lambda = (3, \overline{1})$, then the second overpartition rank of $\lambda$ is given by $1-1=0$.

We also introduce a variation of the overpartition bracket corresponding to $\overline{r}_2(\lambda)$. Given an overpartition $\lambda$ into odd parts, the \emph{second overpartition bracket} of $\lambda$ is  the length of the longest substring of $\lambda$ of the form $(\ell_1, \ell_2, \dots, \ell_k)$, where for all $1\leq i < k$, one of the following holds:
\begin{itemize}
	\item $\ell_i = \ell_{i+1}$

	\item $\ell_i = \ell_{i+1} +2$  and at least one of $\ell_i$ or $\ell_{i+1}$ is overlined.
\end{itemize} 
 We denote the second overpartition bracket of $\lambda$ by $\os_2(\lambda)$.

For example, if $\lambda =(5,\overline{3},3,1)$, the substrings we consider are
\begin{align*}
	\begin{matrix}
		(5), &
		(5,\overline{3}), &
		(5,\overline{3},3),
	\end{matrix}
\end{align*}
the longest of which has length $3$. Therefore, $\os_2(\lambda)=3$. We see how the second overpartition rank and the second overpartition bracket relate to \eqref{2theorem} in the following lemma.

\begin{lem} \label{2frob2lemma}
Fix nonnegative integers $1 \leq s \leq t$. The coefficient of $z^mq^n$ in
\begin{align*}
\frac{\aqprod{-1}{q^2}{t}q^t}{\aqprod{zq^{2s}}{q^2}{t-s+1}}
\end{align*}
is equal to the number of overpartitions $\lambda$ of $n$  into $t$ odd parts with $\overline{r}(\lambda)=m$ and $\os_2(\lambda)\geq s$.
\end{lem}

The proof of Lemma \ref{2frob2lemma} is almost identical to that of Lemma \ref{overrun}.
We can now give a combinatorial interpretation of \eqref{2theorem} in terms of a second family of buffered Frobenius representations.

\subsection{Buffered Frobenius Representations of the Second Kind}

\begin{defn}
A buffered Frobenius representation of the second kind, or a $B_2$-representation, is a buffered Frobenius representation
\begin{align*}
	\nu
	\in
	\begin{pmatrix}
		\nkand{A}{k}\\
		\nkand{B}{k}
	\end{pmatrix}
\end{align*}
where
\begin{enumerate}
\item $A_1$ is the set of nonempty overpartitions $\alpha_1$ into odd parts.

\item $A_2$ is the set of nonempty partitions $\alpha_2$ into even parts, with $\#(\alpha_2) \leq \os_2(\alpha_1)$.

\item For all $3<i\leq k$, $A_i$ is the set of nonempty partitions $\alpha_i$ into even parts with $\#(\alpha_i)$ less than or equal to the number of occurrences of the largest part of $\alpha_{i-1}$

\item $B_1$ is the set of partitions $\beta_1$ into $\#(\alpha_1)$ nonnegative parts where odd parts may not repeat, with $\sigma_2(\beta_1) \geq \#(\alpha_2)$.

\item For all $2\leq i < k$, $B_i$ is the set of partitions $\beta_i$ into $\#(\alpha_i)$ nonnegative even parts and at most $\#(\alpha_{i+1})$ occurrences of their largest part.

\item $B_k$ is the set of partitions $\beta_i$ into $\#(\alpha_i)$ nonnegative even parts.
\end{enumerate}
We also define the empty array to be a $B_2$-representation with $k=0$.
\end{defn} 

For example, consider the array

\begin{align} \label{frob2ex}
\nu=
\begin{pmatrix}
\widehat{(3,\overline{1})}&(2,2)&(4)\\
(6,5)&(2,0)&(2)
\end{pmatrix}.
\end{align}
On the top row, $\alpha_1$ is an overpartition into odd parts, which satisfies (1). Next, $\alpha_2$ is a partition into two even parts, with two occurrences of its largest part.
Because $\os_2(\alpha_1) = 2$, this satisfies (2). 
Finally, $\alpha_3$ is an partition into a single even part.
Because $\alpha_2$ has two occurrences of its largest part,  this satisfies (3).

On the bottom row, $\beta_1$ is a partition into two parts with no repeating odd parts, and $\sigma_2(\beta_1) = 2$, which satisfies (4). 
Next, $\beta_2$ is a partition into two nonnegative even parts with a single occurrence of its largest part, which satisfies (5). 
Finally, $\beta_3$ is a partition into one nonnegative part, which satisfies (6).
Additionally, $\alpha_1$ is marked with a hat.

As in Section \ref{sec:buff}, we see that Lovejoy's second Frobenius representations of overpartitions correspond to the case $k = 1$ above.
For $k > 1$, we can collapse $B_2$-representations using the jigsaw map.
\begin{prop}
	Let $\mathcal{B}_2$ denote the set of $B_2$-representations, and let $\mathcal{F}_2$ denote the set of second Frobenius representations of overpartitions. Then $j:\mathcal{B}_2 \to \mathcal{F}_2$ is a surjective map. 
\end{prop}
Taken with Theorem \ref{M2bi}, we see that every $B_2$-representation $\nu$ corresponds to an overpartition $\lambda$, although this correspondence is many-to-one.
Thus the ranks we will establish to study $\FrobkII$ do not immediately carry over to the set of overpartitions.

\subsection{Ranks of $B_2$-representations}
Recall the definition of $\chi_i$ from Section \ref{first}.
If
\begin{align*}
\nu = \abarray,
\end{align*}
then we define the \emph{first rank} of $\nu$  to be
\begin{align*}
	\rho_2^1(\nu) := \overline{r}_2(\alpha_1) - r_2(\beta_1) + \chi_1(\nu),
\end{align*}
that is, the second overpartition rank of $\alpha_1$ minus the second partition rank of $\beta_1$ plus $\chi_1(\nu)$.
We also define $\rho_2^1(\emptyset) := 0$.

For $2\leq i \leq k$, we define the \emph{$i$th} rank of $\nu$ to be
\begin{align*}
	\rho_2^{i}(\nu) = \left(\frac{\ell(\alpha_i)}{2} - 1\right) - \frac{\ell(\beta_i)}{2} + \chi_i(\nu ),
\end{align*}
which is an integer since $\alpha_i$ and $\beta_i$ have even parts. We also define $\rho_2^i(\nu) := 0$ whenever $\nu$ has fewer than $i$ columns.

For example, let
\begin{align*}
\nu=
\begin{pmatrix}
\widehat{(3,\overline{1})}&(2,2)&(4)\\
(6,5)&(2,0)&(2)
\end{pmatrix}.
\end{align*}
Then
\begin{align*}
	\rho_2^1(\nu) &= (1 - 1) - (3-1) + 1 = -1 \\
	\rho_2^2(\nu) &= (1-1) - 1 + 0 = -1\\
	\rho_2^3(\nu) & = (2-1) -1 + 0 = 0,
\end{align*}
and $\rho_2^i(\nu) = 0$ for $i > 3$.

We now establish $\FrobkII$ as the generating series for the ranks of $B_2$-representations.

\subsection{Generating Series}

Let $\mathcal{B}_2^k$ denote the set of $B_2$-representations with at most $k$ columns,
\begin{align*}
	\mathcal{B}_2^k := 
	\left\{
	\begin{pmatrix}
		\nkand{\alpha}{j}\\
		\nkand{\beta}{j}
	\end{pmatrix}
	\in
	\mathcal{B}_2\
	\middle|\
	j \leq k
	\right\}
	.
\end{align*}

We see the generating series for the $i$th ranks of $B_2$-representations in $\mathcal{B}_2^k$ in the following theorem.

\begin{thm} \label{2term}
The coefficient of $\nkprex{x}{k} q^n$ in
	\begin{align*}
	\sum_{{
		n_1,
		\dots ,
		n_k \geq 0
	}}
	 \frac{\aqprod{-1}{q}{2N_k}}{q^{N_k}} \prod_{i=1}^{k} \frac{(1-x_{k-i+1})(1-x_{k-i+1}^{-1})q^{2N_i}}{\aqprod{x_{k-i+1}q^{2N_{i-1}},x_{k-i+1}^{-1}q^{2N_{i-1}}}{q^2}{n_i+1}}
	\end{align*}
is equal to the number of $B_2$-representations $\nu \in \mathcal{B}_2^k$ such that $|\nu|=n$ and $\rho_2^i(\nu)=m_i$, where the  count is weighted by $(-1)^{h(\nu)}$.

\end{thm}

\begin{proof}

Consider an arbitrary summand of the form
\begin{align*}
		 \frac{\aqprod{-1}{q}{2N_k}}{q^{N_k}} \prod_{i=1}^{k} \frac{(1-x_{k-i+1})(1-x_{k-i+1}^{-1})q^{2N_i}}{\aqprod{x_{k-i+1}q^{2N_{i-1}},x_{k-i+1}^{-1}q^{2N_{i-1}}}{q^2}{n_i+1}}.
\end{align*}
If $n_1 = \cdots = n_k = 0$, then the summand reduces to  $1$, which corresponds to the empty $B_2$-representation $\nu = \emptyset$.
Otherwise, $n_i > 0$ for some $i$.
Let $j$ be the smallest index so that $n_{j} > 0$.
Then the summand reduces to
\begin{align} \label{eq:B2-jsum}
	 \frac{\aqprod{-1}{q}{2N_k}}{q^{N_k}} \prod_{i=j}^{k} \frac{(1-x_{k-i+1})(1-x_{k-i+1}^{-1})q^{2N_i}}{\aqprod{x_{k-i+1}q^{2N_{i-1}},x_{k-i+1}^{-1}q^{2N_{i-1}}}{q^2}{n_i+1}}.
\end{align}
We claim that the coefficient of $\nkprex{x}{k}q^n$ in \eqref{eq:B2-jsum} is equal to the number of $B_2$-representations 
\begin{align*}
	\nu =
	\begin{pmatrix}
		\nkand{\alpha}{k-j+1}\\
		\nkand{\beta}{k-j+1}
	\end{pmatrix}
\end{align*}
where $\#(\alpha_i) = N_{k-i+1}$, such that $|\nu|=n$ and $\rho_2^i(\nu)=m_i$, where the  count is weighted by $(-1)^{h{(\nu)}}$.
Note that 
\begin{align*}
	  (-1)^{h(\nu)} = (-1)^{\sum \chi_i (\nu)}.
\end{align*}

The parts of $\alpha_1$ and $\beta_1$ are generated by the $i=k$ multiplicand, which we write as
\begin{align} \label{eq:B2-1}
	\bigg(
		\frac{(1-x_1) \aqprod{-1}{q^2}{N_k}q^{N_k}}{\aqprod{x_1 q^{2N_{k-1}}}{q^2}{n_k+1}}
	\bigg)
	\bigg(
		\frac{(1-x_1^{-1})\aqprod{-q}{q^2}{N_k}}{\aqprod{x_1^{-1} q^{2N_{k-1}}}{q^2}{n_k+1}}
	\bigg).
\end{align}
We use the fact that $N_k=N_{k-1} + n_k$ to apply Lemmas \ref{2frob1lemma} and \ref{2frob2lemma} with $t=N_k$ and $s=N_{k-1}$.
Then we see that $\alpha_1$ is an overpartition into $N_k$ odd parts with $\os_2(\alpha_1)\geq N_{k-1}$, and $\beta_1$ is a partition into $N_k$ nonnegative parts where odd parts may not repeat with $\sigma_2(\beta_1)\geq N_{k-1}$.
Here, the exponents of $x_1$ and $x_1^{-1}$ track $\oo{r}_2(\alpha_1)$ and $r_2 (\beta_1)$, respectively.

As in the proof of Theorem \ref{1summand}, the term $(1-x_1)(1-x_1^{-1})$ tracks whether or not $\alpha_1$ and $\beta_1$ are marked with a hat.
Thus, the coefficient of $x_1^{m_1} q^n$ in \eqref{eq:B2-1} is equal to the weighted count of of possible columns $(\alpha_1, \beta_1)^T$ in a $B_2$-representation $\nu$ such that $\#(\alpha_1) = N_k$, $\#(\alpha_2) = N_{k-1}$, $n = |\alpha_1| + |\beta_1|$,  and $m_1 = \oo{r}_2(\alpha_1) - r_2(\beta_1) + \chi_1(\nu) = \rho_2^1(\nu)$, where the count is weighted by $(-1)^{\chi_1(\nu)}$.

For $j<i<k$, the parts of $\alpha_{i}$ and $\beta_{i}$ are generated by the $k-i+1$ multiplicand, which we write as
\begin{align} \label{eq:B2-mid}
	\bigg(
		\frac{(1-x_{i})q^{2N_{k-i+1}}}{\aqprod{x_{i}q^{2N_{k-i}}}{q^2}{n_{k-i+1}+1}}
	\bigg)
	\bigg(
		\frac{(1-x_{i}^{-1})}{\aqprod{x_{i}^{-1}q^{2N_{k-i}}}{q^2}{n_{k-i+1}+1}}
	\bigg).
\end{align}
As in the proof of Lemma \ref{2frob1lemma}, \eqref{eq:B2-mid} generates pairs of columns in the tableau for $\alpha_i$ and $\beta_i$. 
We see that $\alpha_i$ is a nonempty partition into $N_{k-i+1}$ even parts with at least $N_{k-i}$ occurrences of its largest part, and $\beta_i$ is a nonempty partition into $N_{k-i+1}$ nonnegative even parts with at least $N_{k-i}$ occurrences of its largest part.
Here, the exponents of $x_i$ and $x_i^{-1}$ track $\tfrac{\ell(\alpha_i)}{2} -1$ and $\tfrac{\ell(\beta_i)}{2}$, respectively.
As with the previous column, entries marked with a hat are tracked by $(1-x_i)(1-x_i^{-1})$.
Thus, the coefficient of $x_i^{m_i} q^n$ in \eqref{eq:B2-mid} is equal to the weighted count of of possible columns $(\alpha_i, \beta_i)^T$ in a $B_2$-representation $\nu$ such that $\#(\alpha_i) = N_{k-i+1}$, $\#(\alpha_{i+1}) = N_{k-i}$, $n = |\alpha_i| + |\beta_i|$,  and 
$$
	m_i = \bigg(\frac{\ell(\alpha_i)}{2} -1\bigg) - \frac{\ell(\beta_i)}{2}  + \chi_1(\nu) = \rho_2^i(\nu),
$$
where the count is weighted by $(-1)^{\chi_i(\nu)}$.

The parts of $\alpha_{k-j+1}$ and $\beta_{k-j+1}$ are generated by the $i=j$ multiplicand
\begin{align*}
	\bigg(
		\frac{(1-x_{k-j+1})q^{2N_j}}{\aqprod{x_{k-j+1} q^{2N_{j-1}}}{q^2}{n_j+1}}
	\bigg)
	\bigg(
		\frac{(1-{x_{k-j+1}^{-1}})}{\aqprod{x_{k-j+1}^{-1}q^{2N_{j-1}}}{q^2}{n_j+1}}
	\bigg).
\end{align*}
By minimality of $j$, we see that $n_1 = \cdots = n_{j-1} = 0$. Thus, $N_{j-1} = 0$, and the multiplicand reduces to
\begin{align} \label{eq:B2-last}
	\bigg(
		\frac{q^{2N_j}}{\aqprod{x_{k-j+1}q^{2}}{q^2}{n_j}}
	\bigg)
	\bigg(
		\frac{1}{\aqprod{x_{k-j+1}^{-1}q^{2}}{q^2}{n_j}}
	\bigg).
\end{align}
This reflects the fact that neither $\alpha_{k-j+1}$ or $\beta_{k-j+1}$ can be marked with a hat.
As with the previous column, we see that the coefficient of $x_{k-j+1}^{m_{k-j+1}} q^n$ in \eqref{eq:B2-last} is equal to the weighted count of possible columns $(\alpha_{k-j+1}, \beta_{k-j+1})^T$ of $\nu$ such that $\# (\alpha_{k-j+1}) = N_{k-i+1}$, $n = |\alpha_{k-j+1}| + |\beta_{k-j+1}|$, and  $m_{k-j+1} = \rho_2^{k-j+1} (\nu)$, where the count is weighted by $(-1)^{\chi_{k-j+1}(\nu)}$.

By combining these terms, we have counted all possible $\nu \in \mathcal{B}_2^k$   with $|\alpha_i| = N_{k-i+1}$, $|\nu| = n$, $\rho_2^i(\nu) = m_i$, and $h(\nu)$ entries marked with a hat, where the count is weighted by $(-1)^{h(\nu)}$.
By summing over all values of $\nklist{n}{k}$, we count all possible $B_2$-representations in $ \mathcal{B}_2^k$.

\end{proof}

\subsection{Full Rank and Proof of Theorem \ref{thm2}}

As in Section \ref{first}, we define the \emph{full rank} of a $B_2$-representation $\nu$ to be
the sum of the $i$th ranks of $\nu$, 

\begin{align*}
	\rho_2(\nu) := \sum_{i \geq 1} \rho_2^i(\nu).
\end{align*}

\noindent This sum converges for any $B_2$-representation $\nu$, as all but finitely many of the summands vanish. We may now prove Theorem \ref{thm2}.

\begin{proof}[Proof of Theorem \ref{thm2}]
Let $\zeta_k$ be a primitive $k$th root of unity. The desired generating series, 
\begin{align*}
	\sum_{\nu \in \mathcal{B}^k_2} (-1)^{h(\nu)} \prod_{i=1}^k\zeta_k^{ (i-1)\rho_2^i(\nu)} z^{\tfrac{\rho_2(\nu)}{k}} q^{|\nu|},
\end{align*}
is given by 
\begin{align} \label{punchline2}
\overline{R2}_k(\sqrt[k]{z},\zeta_k\sqrt[k]{z},\dots, \zeta_k^{k-1}\sqrt[k]{z};q) = \RkII.
\end{align}

\end{proof}

We now have our combinatorial interpretation of $\RkII$.
 As in Section \ref{first}, the weighted count in \eqref{punchline2} must vanish for $B_2$-representations whose full rank is not a multiple of $k$.

We close this section by discussing conjugation maps on $\mathcal{B}_2^k$.

\subsection{Conjugation}

Given a $B_2$-representation
\begin{align*}
\nu = \abarray,
\end{align*}
we define $k$ different conjugation maps corresponding to the columns of $\nu$.
To perform the \emph{first conjugation}, we  subtract $1$ from each part of $\alpha_1$ and reverse Algorithm \ref{alg:JS} to obtain a partition into nonnegative even parts $\lambda$ and a partition into distinct even parts $\mu$.
We reverse Algorithm \ref{alg:JS2} on $\beta_1$ and obtain a partition into nonnegative even parts $\gamma$ and a partition into distinct odd parts $\delta$. Note that $\#(\lambda) = \#(\gamma)$ by construction.

We then perform Algorithm \ref{alg:JS} on $\gamma$ and $\mu$ to produce $\alpha_1'$ and perform Algorithm \ref{alg:JS2} on $\lambda$ and $\delta$ to produce $\beta_1'$. 
Next, add $1$ to each part of $\alpha_1'$.
Finally,  mark $\alpha_1'$ with a hat if and only if $\beta_1$ was marked with a hat, and vice versa.
We call
\begin{align*}
	\phi_2^1(\nu) :=
	\begin{pmatrix}
		\alpha'_1 & \alpha_2 & \dots & \alpha_k\\
		\beta'_1 & \beta_2 & \dots & \beta_k
	\end{pmatrix}
\end{align*}
the \emph{first conjugate} of $\nu$. 

For example, if
\begin{align*}
\nu=
	\begin{pmatrix}
		\widehat{(3,\overline{1})}&(2,2)&(4)\\
		(6,5)&(2,0)&(2)
	\end{pmatrix},
\end{align*}
then we see that
\begin{align*}
	\lambda &= (0, 0)\\
	\mu &= (2)\\
	\gamma &= (4,4)\\
	\delta &= (3).
\end{align*}
Performing Algorithms \ref{alg:JS} and \ref{alg:JS2}, produces
\begin{align*}
	\lambda_1' &= (6, \oo{4} )\\
	\mu_1' &= (2, 1 ),
\end{align*}
and adding $1$ to each part of $\lambda'_1$ yields
\begin{align*}
	\phi_2^1(\nu) =
	\begin{pmatrix}
		(7, \oo{5}) &(2,2)&(4)\\
		\widehat{(2, 1)} &(2,0)&(2)
	\end{pmatrix}.
\end{align*}

For $1<i\leq k$, the \emph{$i$th conjugation map} is performed as follows.
First, subtract 2 from each part of $\alpha_i$ to produce $\alpha_i'$, and add 2 to each part of $\beta_i$ to produce $\beta_i'$.
Mark $\alpha_i'$ with a hat if and only if $\beta_i$ was marked with a hat, and vice versa.
We call
\begin{align*}
	\phi_i(\nu) :=
	\begin{pmatrix}
	\alpha_1 &  \dots & \alpha_{i-1} & \alpha_i' & \alpha_{i+1} & \dots & \alpha_k\\
	\beta_1 &  \dots & \beta_{i-1}&  \beta_i' & \beta_{i+1} & \dots & \beta_k
	\end{pmatrix}
\end{align*}
the \emph{$i$th conjugate} of $\nu$.
Keeping $\nu$ as above, we have
\begin{align*}
	\phi_2^2(\nu) & =
	\begin{pmatrix}
		\widehat{(3,\overline{1})}&(4,2)&(4)\\
		(6,5)&(0,0)&(2)
	\end{pmatrix},	
	\\
	\phi_2^3(\nu) &=
	\begin{pmatrix}
		\widehat{(3,\overline{1})}&(2,2)&(4)\\
		(6,5)&(2,0)&(2)
	\end{pmatrix}
	.
\end{align*}

Each of the $i$th conjugation maps exchange the roles of
\begin{align*}
\frac{1-x_i}{\aqprod{x_i q^{2N_{k-i}}}{q^2}{n_{k-i+1}+1}}
\mbox{ and }
\frac{1-x_i^{-1}}{\aqprod{x_i^{-1} q^{2N_{k-i}}}{q^2}{n_{k-i+1}+1}}
\end{align*}
in \eqref{2theorem}. We find the same relations between conjugation maps as in Section \ref{first}.

\begin{prop}
For all $i\geq 1$, we have $\rho_2^i(\phi_2^i(\nu)) = -\rho_2^i(\nu)$.
\end{prop}

\begin{prop}
For all nonnegative integers $i$ and $j$, $\phi_2^i\phi_2^j=\phi_2^j\phi_2^i$.
\end{prop}
Finally, if we define the \emph{full conjugation} to be
\begin{align*}
	\phi_2 := \prod_{i \geq 1} \phi_2^i,
\end{align*}
then $\phi_2$ is defined for all $\nu \in \mathcal{B}_2$, and $\rho_2(\phi_2(\nu)) = -\rho_2(\nu)$.

This concludes our results.

\section{Conclusion}
\label{end}

We began with the series $\Rk$ and $\RkII$, which arose from observing a pattern between the generating series of the Dyson ranks and $M_2$-ranks of overpartitions, and asked whether these new series related to the ranks of overpartitions.
By generalizing the notion of Frobenius representations of overpartitions, we found that $\Rk$ and $\RkII$ are weighted generating series for the full ranks of buffered Frobenius representations, which lie over the set of overpartitions and generalize the first and second Frobenius representations of overpartitions.
It is somewhat disappointing then that  the full rank functions are not well defined on the set of overpartitions -- compare for example
\begin{align*}
	\rho_1\left(
	\begin{pmatrix}
		\widehat{(3,3,2,1)} & (1, 0, 0 )\\
		(\oo{3},\oo{2},2,2) & (4, 1, 1)
	\end{pmatrix}
	\right)
\mbox{ and }
	\rho_1\left(
	\begin{pmatrix}
		(3,3,2,1) & (1, 0, 0 )\\
		(\oo{3},\oo{2},2,2) & (4, 1, 1)
	\end{pmatrix}
	\right).
\end{align*}
Note that the full conjugation maps are well-defined.
That is, $j(\phi_\alpha(\nu)) = j(\phi_\alpha(\nu'))$ whenever $j(\nu) = j(\nu')$, for $\alpha = 1, 2$.
Additionally, it not immediately clear why a sum weighted by roots of unity should produce a meaningful count.

One would hope that there exists a family of ``$M_k$-ranks'' of overpartitions, whose generating series are given by
\begin{multline} \label{eq:Mk}
	\sum_{n \geq 0} \sum_{m \in \ZZ} \oo{N[k]}(m,n) z^m q^n \\
	= \ovr \bigg( 1+ 2 \sum_{n=1}^\infty \frac{(1-z)(1-z^{-1})(-1)^nq^{n^2+kn}}{(1-zq^{kn})(1-z^{-1}q^{kn})} \bigg).
\end{multline}
By setting $z = 1$ in \eqref{eq:Mk}, we at least have that
\begin{align} \label{eq:endcount}
	\sum_{m \in \ZZ} \oo{N[k]}(m,n) = \oo{p}(n),
\end{align}
as expected.
It seems likely that the coefficients $\oo{N[k]}(m,n)$ are nonnegative integers, which remains open.

It is sufficient that an $M_k$-rank candidate satisfy
\begin{align*}
\sum_{n \geq 0} \overline{N[k]}(m,n)q^n  
=2 \ovr \sum_{n \geq 1} \frac{(-1)^{n+1} q^{n^2+k|m|n}(1-q^{kn})}{(1+q^{kn})},
\end{align*}
which is a generalization of Proposition 3.2 \cite{LovejoyD} and Corollary 1.3 \cite{LovejoyM2}.
We see an avenue for this work via the two interpretations of $\oo{R[2]}(z;q)$ as both the generating series of the $M_2$-ranks of overpartitions, and as the weighted generating series of the full ranks of $B_1$-representations in $\mathcal{B}_1^2$.
One might wonder if the parity of $k$ determines behavior in $\Rk$.
Perhaps understanding how to map $\mathcal{B}_1^2 \to \mathcal{F}_2$ will shed light on how to treat the rest of the $\mathcal{B}_1^k$ and $B_2^k$.
Alternatively, there may be a ``$k$th Frobenius representation'' of overpartitions closer in spirit to Lovejoy's work.

Of course, we should be interested in determining the congruences arising from any rank-like function. We may be able to use \eqref{eq:endcount} to move from congruences of buffered Frobenius representations back to congruencies of overpartitions.

There is also the question of analytics to consider.
Since the series $\Rk$ and $\RkII$ are related to overpartition ranks, and can be obtained from the $q$-hypergeometric series, it is natural to ask if these series exhibit any modular properties.
This could be investigated separately of establishing a higher $M_k$-rank.


\subsection*{Acknowledgments}
The author is very grateful to the referee for uncovering multiple errors and suggesting improvements in the presentation of the results, to Jeremy Lovejoy for careful reading of earlier drafts and many helpful comments, and to Thomas Schmidt for a useful observation for future work.


\begin{thebibliography}{10}

\bibitem{AndrewsWW}
George~E. Andrews.
\newblock Problems and prospects for basic hypergeometric functions.
\newblock In {\em Theory and application of special functions ({P}roc.
  {A}dvanced {S}em., {M}ath. {R}es. {C}enter, {U}niv. {W}isconsin, {M}adison,
  {W}is., 1975)}, pages 191--224. Math. Res. Center, Univ. Wisconsin, Publ. No.
  35. Academic Press, New York, 1975.

\bibitem{Abook}
George~E. Andrews.
\newblock {\em The theory of partitions}.
\newblock Addison-Wesley Publishing Co., Reading, Mass.-London-Amsterdam, 1976.
\newblock Encyclopedia of Mathematics and its Applications, Vol. 2.

\bibitem{AndrewsF}
George~E. Andrews.
\newblock Generalized {F}robenius partitions.
\newblock {\em Mem. Amer. Math. Soc.}, 49(301):iv+44, 1984.

\bibitem{AGCrank}
George~E. Andrews and F.~G. Garvan.
\newblock Dyson's crank of a partition.
\newblock {\em Bull. Amer. Math. Soc. (N.S.)}, 18(2):167--171, 1988.

\bibitem{ASDrank}
A.~O.~L. Atkin and P.~Swinnerton-Dyer.
\newblock Some properties of partitions.
\newblock {\em Proc. London Math. Soc. (3)}, 4:84--106, 1954.

\bibitem{Berk}
Alexander Berkovich and Frank~G. Garvan.
\newblock Some observations on {D}yson's new symmetries of partitions.
\newblock {\em J. Combin. Theory Ser. A}, 100(1):61--93, 2002.

\bibitem{CLFrob}
Sylvie Corteel and Jeremy Lovejoy.
\newblock Frobenius partitions and the combinatorics of {R}amanujan's
  {$_1\psi_1$} summation.
\newblock {\em J. Combin. Theory Ser. A}, 97(1):177--183, 2002.

\bibitem{Opartns}
Sylvie Corteel and Jeremy Lovejoy.
\newblock Overpartitions.
\newblock {\em Trans. Amer. Math. Soc.}, 356(4):1623--1635, 2004.

\bibitem{Dyson}
F.~J. Dyson.
\newblock Some guesses in the theory of partitions.
\newblock {\em Eureka}, (8):10--15, 1944.

\bibitem{GasperRakhman}
George Gasper and Mizan Rahman.
\newblock {\em Basic hypergeometric series}, volume~96 of {\em Encyclopedia of
  Mathematics and its Applications}.
\newblock Cambridge University Press, Cambridge, second edition, 2004.
\newblock With a foreword by Richard Askey.

\bibitem{Joichi}
J.~T. Joichi and Dennis Stanton.
\newblock Bijective proofs of basic hypergeometric series identities.
\newblock {\em Pacific J. Math.}, 127(1):103--120, 1987.

\bibitem{LovejoyD}
Jeremy Lovejoy.
\newblock Rank and conjugation for the {F}robenius representation of an
  overpartition.
\newblock {\em Ann. Comb.}, 9(3):321--334, 2005.

\bibitem{LovejoyM2}
Jeremy Lovejoy.
\newblock Rank and conjugation for a second {F}robenius representation of an
  overpartition.
\newblock {\em Ann. Comb.}, 12(1):101--113, 2008.

\end{thebibliography}

\end{document}